\documentclass[12pt]{article}
\usepackage[utf8]{inputenc}
\usepackage[margin=2.5cm]{geometry}
\usepackage{amsfonts,amsthm}
\usepackage[dvipsnames]{xcolor}
\newtheorem{prop}{Proposition}[section]
\newtheorem{lemma}{Lemma}[section]
\newtheorem{thm}{Theorem}[section]
\newtheorem{corollary}{Corollary}[section]
\newtheorem{defn}{Definition}[section]
\newtheorem{remark}{Remark}[section]
\usepackage{amssymb,amsmath}
\usepackage{comment}
\newcommand{\norm}[1]{\left\lVert#1\right\rVert}
\newcommand{\inprod}[2]{\left<#1,#2 \right>}
\newcommand{\R}{\mathbb{R}}
\newcommand{\e}{\varepsilon}
\numberwithin{equation}{section}

\title{Density-valued solutions for the Boltzmann-Enskog process}
\author{Christian Ennis, Barbara R\"{u}diger, and Padmanabhan Sundar}

\begin{document}

\maketitle
\begin{abstract}
    The time evolution of moderately dense gas evolving in vacuum described by the Boltzmann-Enskog equation is studied. The associated stochastic process, the Boltzmann-Enskog process, was constructed in \cite{ABS} and further studied in \cite{sund},\cite{sund_uniqueness}. The process is given by the solution of a McKean-Vlasov equation driven by a Poisson random measure,  the compensator depending on the distribution of the solution \cite{ABS},\cite{sund}. The existence of a marginal probability density function at each time for the measure-valued solution is established here by using a functional-analytic criterion in Besov spaces \cite{deb-rom},\cite{four}. In addition to existence, the density is shown to reside in a Besov space. The support of the velocity marginal distribution is shown to be the whole of $\mathbb{R}^3$. 
\end{abstract}
\section{Introduction}
The Boltzmann equation describes the time evolution of the density function in position-velocity space for a classical particle subjected to possible collisions by other particles in a diluted gas \cite{Bo}, \cite{cerc} that expands in vacuum for a given initial distribution.  The Boltzmann equation forms the basis for the kinetic theory of gases.
\par
The density function $f(t, r, v)$ depends on time $t \ge 0$, the location variable $r \in \R^d$, and the velocity variable $v \in \R^d$ of the point particle. The Boltzmann equation is given by
\begin{equation}
\frac{ \partial f}{\partial t} (t, r, v)+  v \cdot \nabla_r f(t, r, v) = \mathcal{Q}(f, f)(t,r,v), \label{eqn1.1}
\end{equation}
where $\mathcal{Q}$, known as the collision operator, is a certain quadratic form in $f$.

\par 
In $\R^3$, set $\Lambda := \R^3\times (0, \pi] \times [0, 2\pi)$. Then $\mathcal{Q}$ can be written in the general form 
\begin{equation}
\mathcal{Q}(f, f) (t,r,v)= \int_{\Lambda} \{f(t, r, v^\star) f(t, r, u^\star)-f(t, r, v) f(t, r, u)\}B(v, du, d\theta)d\phi. \label{eqn1.2}
\end{equation}
Any particle travels along a straight line until an elastic collision occurs with another particle.  Each $v \in \R^3$ in (\ref{eqn1.2}) denotes the velocity of an incoming  particle which may hit, at the fixed location $r\in \mathbb{R}^3$,  particles whose velocity is fixed as $u\in \mathbb{R}^3$. 
Let $u^\star \in \R^3$ and $v^\star \in \R^3$ denote the resulting outgoing velocities (upon collision) corresponding to the incoming velocities $u$ and $v$ respectively. $\theta$ $\in (0,\pi]$ denotes the azimuthal or colatitude angle of the deflected velocity, $v^{\star}$ (see \cite{tanaka}). Having determined $\theta$, 
the longitude angle $\phi \in [0, 2\pi)$ measures in polar coordinates,  the location of  $v^*$, and hence  that of $u^\star$, as explained below. 
 
\noindent In the Boltzmann model, as the collisions are assumed to be elastic,  conservation of kinetic energy as well as momentum of the molecules holds. That is, considering particles of mass $m =1$, the following equalities hold:
\begin{equation}
\left\{
\begin{aligned}
u^\star+v^\star =u+v\\
|u^\star|^2+|v^\star|^2 =|u|^2+|v|^2 
\end{aligned}
\right. \label{CONSERVATION}
\end{equation} 
\begin{equation} 
\left\{
\begin{aligned}
v^\star=v+ ({\bf n},u-v) {\bf n}\\
u^\star=u-({\bf n},u-v){\bf n}
\end{aligned}
\right. \label{COLLISION}
\end{equation}
   
\noindent where 
\begin {equation}\label{eqn1.5}
{\bf n}=\frac{v^\star-v}{|v^\star-v|}.
\end {equation}
The collision dynamics are reversible since the Jacobian of the transformation (\ref{CONSERVATION}) has determinant 1 and  $(u^\star)^\star=u$. 

In the velocity sphere, the post-collision velocity  $u^*$ can be written in terms of  $\theta\in (0,\pi]$ measured from the center, and longitude angle $\phi \in  [0,2\pi)$ of  the deflection  vector ${\bf n}$ in the sphere with northpole $u$ and southpole $v$ centered at $\frac{u+v}{2}$, which are used in equation (\ref{eqn1.1}) and \eqref{eqn1.2} (see  e.g. \cite{Br},  \cite{HK}, \cite{tanaka87}).

\par 
$B(v, du, d\theta)$ is a probability kernel such that it is $\sigma$-additive positive measure defined on the Borel $\sigma$-field ${\mathcal B}(\R^3) \times {\mathcal B}((0, \pi])$ for each fixed $u \in \R^3$. The form of $B$ depends on the version of the Boltzmann equation. In Boltzmann's work 
 \cite{Bo}, 
\begin{align}
B(v, du, d\theta) &= |(u-v)\cdot n| du d\theta \notag\\
& = |u - v| du \cos \left(\frac{\theta}{2}\right) \sin \left(\frac{\theta}{2}\right) d\theta. \label{eqn1.6}
\end{align}
In the case where the molecules interact by a force which varies as the $n$th inverse power of the distance between their centers, one has 
\cite{cerc}, 
\begin{equation}
B(v, du, d\theta) = |(u-v)|^{\frac{n-5}{n-1}} b(\theta)dud\theta \label{eqn1.7}
\end{equation}
where $b$ is a measurable positive function of $\theta$. In particular, for $n = 5$, one has the case of ``Maxwellian molecules", where 
\[
B(v, du, d\theta) = b(\theta) du d\theta.
\]
The function $b(\theta)$ decreases and behaves like $\theta^{- 3/2}$ for $\theta \downarrow 0$, see, for e.g. \cite{cerc}. Note that 
in the latter case $\int_0^{\pi}b(\theta)d\theta =+ \infty$. 
\par In the present paper, we take
\begin{equation}
B(v, du, d\theta) = \sigma (|u - v|) du Q(d\theta) \label{eqn1.8}
\end{equation}
where  $Q$ is a $\sigma$-finite measure on ${\mathcal B}((0, \pi])$. Specifically, let $\nu \in (0,1)$. We assume that there are $0 < c < C$ such that
\begin{equation}
\label{defn-of-Q-measure}
    Q(d\theta) = b(\theta)d\theta \,\text{ with }\, c\theta^{-1-\nu} < b(\theta) < C\theta^{-1-\nu}.
\end{equation} 
When $Q(d\theta)$ is taken to be integrable, one speaks of the cutoff case. Thus, the system studied has no angular cutoff such as in \cite{Ar4}. Notice that $\nu = \frac{1}{2}$ corresponds to Maxwellian molecules for the appropriate $\gamma$. Friesen, R\"{u}diger, and Sundar  \cite{sund} consider a more general $Q$ with no explicit form, simply that $Q$ is $\sigma$-finite with $\int_0^\pi \theta Q(d\theta) < \infty$. Their results still apply in this case, as they note in their paper (see Example 1.1 in \cite{sund}) the work applies for the case of hard potentials and long-range interactions. The non-negative, continuous function $\sigma(z)$ behaves similar to $ (1 + z)^\gamma$ for $z \in \R^+$ and $0 < \gamma < 2$. Specifically, we assume there exists $\gamma \in (0,2]$ and $c_\sigma \geq 1$ such that 
\begin{align}
\label{sigma-lip}
    |\sigma(|z|) - \sigma(|w|)| \leq c_\sigma||z|^\gamma - |w|^\gamma|\,\,\,z,w \in \R^3\setminus\{0\}
\end{align}
and 
\begin{align}
    \label{sigma-gamma}
    \sigma(|z|) \leq c_\sigma
        (1 + |z|^2)^\frac{\gamma}{2}
\end{align}
Note that by concavity, the following bound may be used in calculations.
    \begin{align}
        \sigma(|z|) \leq c_{\sigma}(1 + |z|^\gamma)
    \end{align}
Morgenstern \cite{Mo} ``mollified" $Q(f, f)$ by replacing it by 
\begin{align*}
&Q_M(f, f)(t,r,u) \\
& = \int \{f(t, r, v^\star) f(t, q, u^\star)-f(t, r, v) f(t, q, u)\}K_M(r, q)B(v, du, d\theta) dq d\phi
\end{align*}
with some measurable $K_M$ and $B$ such that $K_M(r, q)B(v, du, d\theta)$ has a bounded density with respect to the Lebesgue measure on $\R^3 \times [0, \pi]$, and obtaining a global existence theorem in $L^1(\R^3 \times \R^3)$. Povzner \cite{Pov} obtained existence in the space of Borel measures in $x$, with the term $K_M(r, q)B(v, du, d\theta) dq d\phi$ replaced by $K_P(r-q, u-v)du dq$. Povzner's model considers linear growth in $u - v$ and being bounded in position and vanishing for large $r - q$, which is equivalent to the case of $\gamma = 1$. Povzner determines uniqueness for strong solutions under separate moment estimates. 
\par
According to \cite{cerc} (p. 399), this modification of Boltzmann's equation by Povzner is ``{\it close to physical reality}". 
Let us now associate to \eqref{eqn1.1}, \eqref{eqn1.2} its weak (in the functional analytic sense) version. The following proposition by 
 Tanaka  \cite {tanaka}  is important for the weak formulation of the equation.
\begin{prop}
\label {PropAppTanaka}  
Let $\Psi(r,v) \in C_c(\mathbb{R}^6)$, as a function of $r,v \in \mathbb{R}^3$.   With $B$ as in \eqref{eqn1.8}, we have
\begin{equation} 
\begin{split}
&\int_{\mathbb{R}^9\times (0,\pi]\times [0,2 \pi)} \Psi(r,v)f(t, r,v^\star) f(t, r, u^\star) B(v, du, d\theta) dr dv  d\phi
\\&= \int_{\mathbb{R}^9\times (0,\pi]\times [0,2 \pi)} \Psi(r,v^\star)f(t, r, v) f(t, r, u) B(v, du, d\theta) dr dv d\phi \label{eqn1.9}
\end{split}
\end{equation}
\end{prop}
Consider the Boltzmann equation \eqref{eqn1.1} with collision operator \eqref{eqn1.2}. We multiply \eqref{eqn1.1} by a function  $\psi$  (of $(r, v) \in \mathbb{R}^6$)
 belonging to  $C^1_0(\mathbb{R}^6)$, 
and integrate with respect to $r$ and $v$. Using integration by parts and Proposition \ref{PropAppTanaka}, we arrive at the weak form of the Boltzmann equation: 
\begin{align}
&\int_{\mathbb{R}^6} \psi(r,v) \frac{ \partial f}{\partial t} (t, r, v)drdv - \int_{\mathbb{R}^6} f(t, r, v)(v,  \nabla_r \psi(r,v)) drdv \notag\\
&= \int_{\mathbb{R}^6} f(t, r, v) L_f\psi(r,v) drdv \label{WBE}
\end{align}
 for all $t\in \mathbb{R}_+$  with
\[
L_f\psi(r,v)= \int_{\mathbb{R}^3\times (0,\pi]\times [0,2 \pi)} \{\psi(r,v^\star)- \psi(r,v)\}f(t, r, u)B(v, du, d\theta)d\phi,
\]
where $B$ is as in \eqref{eqn1.8}.

\par
To proceed further, let us introduce an  approximation to the weak form of the Boltzmann equation by introducing a smooth real-valued, non-negative function $\beta$ with compact support defined on $\R^1$. Then (\ref{WBE}) is instead: 
 \begin{align}
& \int_{\mathbb{R}^6} \psi(r,v) \frac{ \partial f}{\partial t} (t, r, v)drdv - \int_{\mathbb{R}^6} f(t, r, v)(v,  \nabla_r \psi(r,v)) drdv\notag \\
 & = \int_{\mathbb{R}^6} f(t, r, v) L^{\beta}_f\psi(r,v) drdv 
 \label{WBEA} 
 \end{align}
 for all $\psi \in C^1_0(\mathbb{R}^6)$ and for all $t\in \mathbb{R}_+$, with
\[
L_f^\beta\psi(r,v) = \int_{\mathbb{R}^6\times (0,\pi]\times [0,2 \pi)} \{\psi(r,v^\star)- \psi(r,v)\}f(t, q, u) \beta(|r - q|)dqB(v, du, d\theta) d\phi . 
\]
Heuristically, when $\beta \to \delta_0$, then any solution of \eqref{WBEA} tends to a solution of Boltzmann's equation \eqref{WBE}, so that 
$\beta$ can be seen as a regularization for \eqref{WBE}. 
\par
Equation \eqref{WBEA} is thus the (functional analytic) weak form of an equation closely related to the Boltzmann equation, which can be written as 
\begin{equation}
\frac{\partial f}{\partial t} (t, r, v) + v \cdot \nabla_rf(t, r, v) = Q_E^{\beta}(f, f) (t,r,v), \label{EQ:03}
\end{equation}
with 
\begin{align*}
&Q_E^{\beta}(f, f) (t,r,v) \\
& = \int_{\Lambda} \int_{\R^3} \{f(t, r, v^*)f(t, q, u^*) - f(t, r, v)f(t, q,u)\} \beta(|r - q|)dq B(v, du, d\theta)d\phi.
\end{align*}
In the case where $\beta$ is replaced by the characteristic function (or a smooth version of it like in \cite{cerc}) of a ball of radius $\epsilon > 0$, this is Boltzmann-Enskog's equation used for (moderately) ``dense gases" taking into account interactions at distance $\epsilon$ between molecules. In its physical interpretation particles have a small radius and are allowed to overlap during an elastic collision.
\par
If $\mu_t$ denotes the Borel probability measure on $\R^6$ corresponding to a smooth density function $f(t,q, u)$, 
i.e.
\[
\mu_t(dq, du) = f(t, q, u) dq du,
\] 
then the equation \eqref{WBEA} can be 
written as
\begin{equation}
\frac{\partial}{\partial t} \langle \mu_t, \psi \rangle - \langle \mu_t, (v, \nabla_r \psi(r,v))\rangle = \langle \mu_t, L^{\beta}_{\mu_t}\psi \rangle \label{MtWBEA}
\end{equation}
where 
\[
L^{\beta}_{\mu_t}\psi(r,v) = \int_{\mathbb{R}^6\times (0,\pi]\times [0,2 \pi)} \{\psi(r,v^\star)- \psi(r,v)\} \beta(|r-q|)\sigma(|v-u|)Q(d\theta)\mu_t(dq, du)d\phi.  
\]
recalling the definition of $B$ \eqref{eqn1.8}, where we group $du$ with $f$. In the above, the sharp bracket $\langle \cdot, \cdot \rangle$ to denotes integration with respect to $\mu_t$.  If $\mu_t$ satisfies \eqref{MtWBEA}, we say that $\mu_t$ is a weak solution of the Boltzmann-Enskog equation.\\

\smallskip \noindent \textbf{Results on the spatially homogeneous Boltzmann Equation}\\
The spatially homogeneous Boltzmann equation is especially well-studied, with a plethora of literature on the solutions of the equation. It is obtained by integrating over space the Boltzmann equation. In his pioneering works \cite{tanaka},\cite{tanaka87} Tanaka provided a probabilistic interpretation and approach to the space-homogeneous Boltzmann equation. His results influenced many future research related to the (spatially homogeneous) Boltzmann equation
%, from which would like to mention
%Horowitz and Karandikar \cite{HK90},Sznitman \cite{S91}, Fournier \cite{F15},Liping \cite{X16}, Fournier and Mischler \cite{FM16}, Cortez and Fontbona \cite{CF18} 
yielding newer and deeper insights into the dynamics of the Boltzmann equation including extensive generalization of Tanaka's original works. In fact, Tanaka  \cite{tanaka} established that the spatially homogeneous case has a weak solution which can be identified as the distribution of a stochastic process. Fournier \cite{four} was able to showcase the existence of a probability density function. The probabilistic interpretations of the spatially homogeneous Boltzmann equation inspire such a treatment of the Boltzmann-Enskog equation. A common overlap between probabilistic interpretations of these different systems is the observation of conservation laws, which in turn, influence the laws of solutions. \\
  One expects solutions to the Boltzmann or Boltzmann-Enskog equation to have is the conservation of momentum and energy. This can be stated deterministically for the particle density function. The relation to the probabilistic interpretation is inherent, provided the solution is a probability density function as well. Let $f_0(r,v) \geq 0$ be the particle density function of the gas at initial time $t = 0$, with the time evolution $f_t = f_t(r,v)$ being obtained from the Boltzmann-Enskog equation (\ref{EQ:03}). Under the conservation laws (\ref{CONSERVATION}) and (\ref{COLLISION}) it is assumed that the solution will satisfy, under few conditions on $\{f_t\}_{t \geq 0}$, the conservation laws below. 
\begin{align}
\label{conv-laws}
    \int_{\R^{6}}f_t(r,v)drdv &= \int_{\R^6}f_0(r,v)drdv \notag\\
    \int_{\R^{6}}vf_t(r,v)drdv &= \int_{\R^6}vf_0(r,v)drdv\\
    \int_{\R^{6}}|v|^2f_t(r,v)drdv &= \int_{\R^6}|v|^2f_0(r,v)drdv \notag
\end{align}
See e.g. the appendix at the end of this article for precise statements.
 This work relies on a stochastic interpretation of the Boltzmann-Enskog system. Albeverio, R\"udiger and Sundar constructed a stochastic process whose law weakly solves the Boltzmann-Enskog equation \cite{ABS}. Under sufficient conditions, they were able to prove its existence in \cite{ABS}, for the case of a collision kernel described by the Maxwell interaction,  and later for other cases in \cite{sund}, including the case of long-range interaction without angular cutoff. Further, moment estimates, dependent on time, were established for these solutions in \cite{sund}. The Boltzmann and Enskog equations are well-studied in the field of gas models. A good overview of the physical interpretation of the Boltzmann equation, as well as some properties known from a mathematical standpoint, can be seen from Cercignani \cite{cerc}. More recent results can be seen in the review paper from Villani \cite{vill}. Our paper will focus on the Boltzmann-Enskog equation (\ref{EQ:03}), which should not be confused with the Boltzmann equation (spatially homogeneous or otherwise). Results for weak solutions will be recalled and discussed, before the main results of this paper are given.\\
The Boltzmann-Enskog process was constructed in \cite{ABS} from the solution of a Fokker-Planck (Kolmogorov forward) equation driven by a jump process. Existence results with higher regularity were determined in \cite{sund}, and uniqueness was studied in \cite{sund_uniqueness}. The main objective of this paper is to establish existence of a density for the velocity marginal of a measure-valued solution to the Boltzmann-Enskog equation. The Boltzmann-Enskog process is utilized so that probabilistic methods can be utilized. Specifically, aspects of additive processes are beneficial in studying the law of the Boltzmann-Enskog process, as their characteristic functions are more easily computable than a general stochastic process. These probabilistic methods are then used for a functional-analytic result \cite{deb-rom}, which gives the main result of the paper.\\
\textbf{Organization of the Paper}\\
This paper is organized as follows. Section 1 introduces the setup of the Boltzmann-Enskog equation and its weak formulation. Section 2 gives preliminary results that are necessary and useful in later sections of this work. An essential theorem in this paper, due to Debussche and Romito \cite{deb-rom} is introduced with a brief discussion on its use by Fournier \cite{four}. The Boltzmann-Enskog process as defined in \cite{sund} is properly introduced as well.
Section 3 introduces an additive process, indexed by $\varepsilon$, which closely approximates the Boltzmann-Enskog process, converging to it in $L^1(\mathbb{R}^6)$. Section 4 establishes support on $\mathbb{R}^3$ for the marginal distribution of the velocity component of the Boltzmann-Enskog process. Section 5 provides useful formulae to additive processes, with proof of stochastic continuity for the velocity component of the approximating process. The characteristic function of the approximating process is studied deterministically. The relevance of these estimates to the distribution of the Boltzmann-Enskog process is established. Section 6 is dedicated to the proof of the main theorem, utilizing results on the approximating process's distribution. 
\section{Preliminaries}
\label{assume}
Equation (\ref{EQ:03}) only makes sense for function-valued solutions. As such, it is necessary to further study the weak formulation \eqref{MtWBEA} of the Boltzmann-Enskog equation. Before doing so, we introduce terms to rewrite $\textbf{n}$ introduced in \eqref{eqn1.5} as sum of vectors $n_1$ and $n_2$ which are parallel and orthogonal to $u - v$ respectively. Recall that $S^d = \{v \in \mathbb{R}^{d+1}: |v| = 1\}$, and define $S^{d-2}(u-v) := \{w \in \mathbb{R}^d: |u-v| = |w|, (u-v,w) = 0\}$. The following lemma is Lemma 2.1 from \cite{sund}, which we write for dimension $d = 3$. 
\begin{lemma}
    Let $u,v \in \mathbb{R}^3$ with $u \neq v$. Let $n \in S^2$. Then there exits $(\varphi,\xi) \in [0,\pi]\times S^1$ and a measurable, bijective function $\Gamma(u-v,\circ): S^1 \to S^1(u-v)$, $\xi \to \Gamma(u-v,\xi)$ such that
    \begin{equation*}
        n = \cos(\varphi)\frac{u-v}{|u-v|} + \sin(\varphi)\frac{\Gamma(u-v,\xi)}{|u-v|}
    \end{equation*}
    where $\varphi$ is the angle between $u - v$ and $n$.
\end{lemma} 
Algebraic and geometric arguments can be made to transform the angle $\varphi$ to give another parameterization that allows better use of the symmetries in the collisions to instead depend on $\theta$ where $\theta$ is the angle between $v - u$ and $v^* - u^*$. The reader is referred to a brief calculation shown on pages 5 and 6 of \cite{sund}. These calculations lead to the another parameterization,
\begin{align}
\label{new-system}
    \begin{cases} v^{\star} &= v + \alpha(v,u,\theta,\xi) \\ u^{\star} &= u - \alpha(v,u,\theta,\xi) \end{cases}
\end{align}
where 
\begin{align}
    \label{alpha-defn}
    \alpha(v,u,\theta,\xi) = \sin^2\left(\frac{\theta}{2}\right)(u-v) + \frac{\sin(\theta)}{2}\Gamma(u-v,\xi)
\end{align}
so that (\ref{new-system}) may be used instead of (\ref{COLLISION}). Note that making the definition $\alpha(v,v,\theta,\xi) = 0$ allows the parameterization to still hold if $v = u$. Recall the equations (\ref{COLLISION}). It has been shown by Tanaka \cite{tanaka} that $(u,v) \to (u-v,n)n$ cannot be smooth. However, he introduced a transformation of parameters, which is bijective with the Jacobian equal to 1, to alleviate this issue. This transformation was generalized to higher dimensions in \cite{four2}. We will introduce this transformation, adapted to our parameterization in the following lemma.  
\begin{lemma}[Lemma 3.1 from \cite{tanaka}]
\label{xi-lemma}There exists a measurable function $\xi_0(x-x',y-y',\xi)$ on $\mathbb{R}^{12}\times S^2$ such that 
\begin{equation}
\label{alph-cont}
    |\alpha(u,v,\theta,\xi) - \alpha(u^*,v^*,\theta,\xi + \xi_0(v-v^*,u-u^*,\xi))| \leq 2\theta(|v-u| + |v^*-u^*|)
\end{equation}
\end{lemma}
In addition to (\ref{alph-cont}), we will utilize the following inequality as well. 
\begin{align}
    \label{alpha-lip}
    |\alpha(v,u,\theta,\xi)| \leq \frac{1}{2}\theta|v-u|
\end{align}
which can be determined as $\Gamma$ is a bijection preserving magnitude and its connection to $\alpha$. More detail on the construction of $\alpha$ can be found in \cite{four}, with a higher dimensional construction given in \cite{sund}, and the references therein.\\
The weak formulation of the Boltzmann-Enskog equation (\ref{MtWBEA}) can now be given using the parameterization (\ref{new-system}).
\\ \textbf{Measure-Valued Solutions to the Boltzmann-Enskog equation}\\
Allow $f_t$ to be a sufficiently smooth solution to (\ref{EQ:03}). Testing (\ref{EQ:03}) against a smooth function $\psi$ and integrating by parts gives
\begin{align}\label{EQ:102}
 &\ \int \limits_{\R^{6}}\psi(r,v)\left( \frac{\partial f_t(r,v)}{\partial t} + v \cdot \nabla_r f_t(r,v)\right)drdv
 \\ \notag &= \frac{d}{dt} \int \limits_{\R^{6}}\psi(r,v) f_t(r,v)dr dv - \int \limits_{\R^{6}} (v \cdot \nabla_r \psi)(r,v) f_t(r,v)drdv.
\end{align}
Likewise,
\begin{align}\label{EQ:101}
 \int \limits_{\R^{6}}\psi(r,v)\mathcal{Q}(f_t,f_t)(r,v)dr dv = \int \limits_{\R^{12}}\sigma(|v-u|)\beta(r-q)(\mathcal{L}\psi)(r,v;u)f_t(r,v)f_t(q,u)drdv dq du,
\end{align}
where $\mathcal{L}\psi$ is, for $\psi \in C^1(\R^{6})$, with $\Xi = (0,\pi] \times S^{1}$ defined by
\[
(\mathcal{L}\psi)(r,v;u) = \int \limits_{\Xi}\left( \psi(r,v + \alpha(v,u,\theta,\xi)) - \psi(r,v)\right) Q(d\theta)d\xi.
\]
For $\psi \in C^1(\R^{6})$, define
\begin{align}
\label{A-operator}
 (\mathcal{A}\psi)(r,v;q,u) &= v \cdot(\nabla_r \psi)(r,v) + \sigma(|v-u|)\beta(r-q)(\mathcal{L}\psi)(r,v;u).
\end{align}
Then combining (\ref{EQ:101}) and (\ref{EQ:102}) with the conservation laws (\ref{conv-laws}), shows that $f_t$ satisfies the weak formulation of the Boltzmann-Enskog equation
\begin{align}\label{Enskog-weak}
 \frac{d}{dt}\int \limits_{\R^{6}} \psi(r,v)f_t(r,v)dr dv =  \int \limits_{\R^{12}} (\mathcal{A}\psi)(r,v;q,u) f_t(r,v)f_t(q,u)drdvdqdu,
\end{align}
Equation (\ref{Enskog-weak}) can then be interpreted in the sense of measure-valued solutions where we instead consider $\mu_{t}(dr,dv) = \int_{\mathbb{R}^6}f_t(r,v)drdv$ and rewrite (\ref{Enskog-weak}) as
\begin{align}
\label{measure-valued-enskog}
\frac{d}{dt}\int_{\mathbb{R}^6}\psi(r,v)\mu_t(dr,dv) = \int_{\mathbb{R}^{12}}(\mathcal{A}\psi)(r,v;q,u)\mu_t(dr,dv)\mu_t(dq,du).   
\end{align}
 The form of equation (\ref{measure-valued-enskog}) suggested in \cite{ABS} to look upon it as a Fokker-Planck equation for jump Markov processes: to wit, consider the processes $(R_t,V_t)$, 
 \begin{align}
     \label{(Rt,Vt)}
       R_t & = R_0 + \int_0^t V_s\,ds\\
       V_t & = V_0 + \int_0^t\int_{\Xi\times\chi\times\mathbb{R}_+}\hat{\alpha}(V_{s-},R_s,u_s,q_s,\theta,\xi,z)dN(s,\theta,\xi,\mu,z)
 \end{align}
 where $\hat{\alpha}(v,r,q,u,\theta,\xi,z) := \alpha(v,u,\theta,\xi)1_{[0,\sigma(|v-u|)\beta(r-q)]}(z)$ and $N$ is a Poisson random measure with compensator
 \begin{equation*}
 d\hat{N}(s,\theta,\xi,\eta,z) = dsQ(d\theta)d\xi d\mu_s(dq,du) dz \text{ on } \mathbb{R}_+\times \Xi \times \chi \times \mathbb{R}_+   
 \end{equation*}
 with $\mu_s$ being the law of $(R_s,V_s)$. For, upon invoking the It\^{o} formula to compute $E(\psi(R_t,V_t))$ for suitable $\psi$, we obtain equation (\ref{measure-valued-enskog}) in its integral form  (see \cite{ABS}). Since $N$ depends on $\mu_\cdot$, we start in an equivalent manner to define the Boltzmann-Enskog process as follows. For a random variable $Z$, we denote its law by $\mathcal{L}(Z)$.
\begin{defn}
\label{Boltzmann-Enskog process}
    Let $\mu_0 \in \mathcal{P}(\mathbb{R}^6)$ be given. A Boltzmann-Enskog process with initial distribution $\mu_0$ consists of the following:
    \begin{itemize}
        \item [a)] A probability space $(\chi,d\eta)$ and a Poisson random measure $N$ with compensator
        \begin{align*}
            d\hat{N}(s,\theta,\xi,\eta,z) = dsQ(d\theta)d\xi d\eta dz \text{ on } \mathbb{R}_+\times \Xi \times \chi \times \mathbb{R}_+
        \end{align*}
        defined on a stochastic basis $(\Omega,\mathcal{F},\mathcal{F}_t,\mathbb{P})$ with the usual conditions.
        \item [b)] A c\`{a}dl\`{a}g process $(q_t,u_t) \in \mathbb{R}^6$ on $(\chi,d\eta)$ and an $\mathcal{F}_t$-adapted c\`{a}dl\`{a}g process $(R_t,V_t)$ on $(\Omega,\mathcal{F},\mathcal{F}_t,\mathbb{P})$ such that $\mathcal{L}(R_t,V_t) = \mathcal{L}(q_t,u_t)$ for any $t \geq 0$, with $\mathcal{L}(R_0,V_0) = \mu_0$. Further, for all $T > 0$
        \begin{align*}
            \int_0^T E\left(|V_t|^{1 + \gamma}\right)dt < \infty
        \end{align*}
        and for $\hat{\alpha}(v,r,q,u,\theta,\xi,z) := \alpha(v,u,\theta,\xi)1_{[0,\sigma(|v-u|)\beta(r-q)]}(z)$,
        \begin{align*}
            \begin{cases}
                R_t & = R_0 + \int_0^t V_s\,ds\\
                V_t & = V_0 + \int_0^t\int_{\Xi\times\chi\times\mathbb{R}_+}\hat{\alpha}(V_{s-},R_s,u_s,q_s,\theta,\xi,z)dN(s,\theta,\xi,\eta,z)
            \end{cases}
        \end{align*}
    \end{itemize}
\end{defn}

\smallskip \noindent {\bf List of notation} \\
Before beginning our work, we first introduce important results established previously in \cite{ABS},\cite{sund}, and \cite{four}. We likewise, define functional spaces and measure spaces utilized in later theorems. Below is a list of spaces we will reference. 
\begin{itemize}
    \item $\mathcal{M}(\mathbb{R}^d)$ is the set of nonnegative finite measures on $\mathbb{R}^d$.
    \item $\mathcal{P}(\mathbb{R}^d)$ is the set of probability measures on $\mathbb{R}^d$.
    \item $\mathcal{P}_p(\mathbb{R}^d)$ is the set of probability measures $\mathbb{P}$ on $\mathbb{R}^d$ such that $\int_{\mathbb{R}^d}|v|^p\,d\mathbb{P} < \infty$.
    \item $C_b(\mathbb{R}^d)$ is the set of real-valued continuous, bounded functions on $\mathbb{R}^d$.
    \item $C^1_c(\mathbb{R}^d)$ is the set of compactly supported, differentiable functions on $\mathbb{R}^d$.
    \item $L^p(\mathbb{R}^d)$ is the set of Lebesgue $p$-integrable functions on $\mathbb{R}^d$, equipped with the usual $L^p$ norm, $\norm{f}_{L^p(\mathbb{R}^d)} := \left(\int_{\mathbb{R}^d}|f(x)|^p\,dx\right)^{1/p}$.
    \item For $s \in (0,1)$, the Besov space $B^s_{1,\infty}$ consists of all functions $f$ such that 
    \begin{equation*}
        \norm{f}_{B^s_{1,\infty}(\mathbb{R}^d)} := \norm{f}_{L^1(\mathbb{R}^d)} + \sup_{h\in\mathbb{R}^d,0<|h|<1}|h|^{-s}\int_{\mathbb{R}^d}|f(x+h) - f(x)|\,dx < \infty
    \end{equation*}
    \item For $\alpha \in (0,1)$, $C^\alpha_b(\mathbb{R}^d)$ is the set of all functions $g$ such that 
    \begin{equation*}
        \norm{g}_{C^\alpha_b(\mathbb{R}^d)} := \sup_{x\in\mathbb{R}^d}|g(x)| + \sup_{x,y\in\mathbb{R}^d\\x\neq y} \frac{|g(x) - g(y)|}{|x-y|^\alpha} < \infty.
    \end{equation*}
\end{itemize}
\textbf{Existence of Density for the Boltzmann Process}\\
With existence of the Boltzmann-Enskog process established in \cite{ABS}, generalized with \cite{sund}, and uniqueness studied in \cite{sund_uniqueness}, there is a good amount of information available on the distribution. The main result of the paper is to establish a density for the velocity marginal distribution of the Boltzmann-Enskog process. For this goal, we will utilize a result made by Debussche and Romito \cite{deb-rom}. This result was utilized by Fournier \cite{four} for the spatially homogeneous Boltzmann equation, where he simplifies the result for a less general case.
\begin{thm}
    \label{deb-rom-exist}
    Let $g \in \mathcal{M}(\R^d)$. Assume there are $0 < \alpha < a < 1$ and a constant $\kappa$ such that for all $\phi \in C^\alpha_b(\R^d)$, all $h \in \R^d$ with $|h| \leq 1$,
    \begin{align}
    \label{deb-rom-eqn}
        \left|\int_{\R^d}\phi(x+h) - \phi(x)g(dx)\right| \leq \kappa\norm{\phi}_{C^\alpha_b(\R^d)}|h|^a
    \end{align}
    Then $g$ has a density in $B^{a-\alpha}_{1,\infty}(\R^d)$ with $\norm{g}_{B^{a-\alpha}_{1,\infty}(\R^d)} \leq g(\R^d) + C_{d,a,\alpha}\kappa$.
\end{thm}
This theorem is Lemma 2.1 in \cite{four}, where the reader is referred to for the proof. It is a straightforward analytical argument, avoiding the need for advanced techniques in functional analysis.\\
In addition to the existence of a density function at each fixed time $t$, we establish Besov norm estimates related to it. It should be noted that different values of $\gamma$ will yield different estimates. Further, each estimate will depend on the finite time $t$. \\
The difficulty in immediate application of similar methods to the spatially homogeneous Boltzmann equation is in the transport term, $v\cdot \nabla_r f_t$ which does not appear in the spatially homogeneous case. However, the transport is not a hindrance in the stochastic interpretation under expectation as one can examine in Definition \ref{Boltzmann-Enskog process}. Still, this will be addressed when establishing deterministic results needed to properly adapt Theorem \ref{deb-rom-exist} for the distribution of the Boltzmann-Enskog process.\\ 
 In \cite{four} N. Fournier proved that a weak solution of a spatially homogeneous Boltzmann equation belongs to some Besov space.
The methods in \cite{four} involved the examination of the spatially homogeneous Boltzmann equation under the case of \textit{hard potentials} where $\nu \in (0,1)$ and $\gamma \in (0,1)$ and \textit{moderately soft potentials} where $\nu \in (0,1)$ and $\gamma \in (-1,0)$. The case of $\gamma \in [1,2)$ is unstudied, but is not immediate under the taken methods as it requires the finiteness of moments of order higher than two, which is not immediate if one solely has the conservation of energy.\\
The existence of density for the Boltzmann-Enskog process is not able to be obtained without an initial hypothesis of higher moments of the distribution, as the conservation of momentum will only guarantee finite moments up to the second moment. The work of the current paper will utilize the established existence of an Boltzmann-Enskog process with higher initial moments \cite{sund} to allow a similar analysis of differences under expectation. While some estimates established in later sections may remain independent of time depending on the specific $\gamma$ taken, this is not stressed often as it is not a concern. The estimates of the paper will only stress a need for remaining finite on a compact time interval.\\
The moment estimates of the Boltzmann-Enskog process are discussed below. Before stating them, we note that they will depend on a fixed time and the initial distribution.
\\ 

\noindent \textbf{Moments of the Boltzmann-Enskog process}\\
The existence result in \cite{sund} required an initial higher order estimate than the usual conservation laws (\ref{conv-laws}), which in turn gave a better regularity of the solution. Their method involved a sequence of empirical measures, establishing tightness and showcasing the limit describes the law of an Boltzmann-Enskog process. This process established that the Boltzmann-Enskog process satisfies moment estimates of order higher than two, dependent on the $\gamma$ connected with the process. The following result is obtained from the proof of Theorem 3.2 and Remark 3.4 in \cite{sund}.
\begin{thm}
\label{sund-estimate-thm}
 Suppose that $\gamma \in (0,2)$ and let $\mu_0 \in \mathcal{P}(\R^{6})$ be such that there exists $\e > 0$ with
 \begin{align}
  \int \limits_{\R^{6}}\left( |r|^{\e} + |v|^{\frac{2}{2-\gamma}\max\{4,1+2\gamma\}} \right)d\mu_0(r,v) < \infty.
 \end{align}
 Then there exists an Boltzmann-Enskog process $(R_\cdot,V_\cdot)$ with distribution $\mu_\cdot$. Moreover, for any $p \geq 4$, there exists a constant $C_p > 0$ such that
 \begin{align}
  E\left( |V_t|^p \right) \leq C_{p}\left( 1 + \int \limits_{\R^{6}}|v|^{\frac{2p}{2-\gamma}}d\mu_0(r,v) \right)  t^{\frac{2p}{2-\gamma}}, t \geq 0,
 \end{align}
 and if $p + \gamma \geq 4$, then also
 \begin{align}
  E\left(\sup \limits_{s \in [0,t]} |V_s|^p\right) \leq C_p \left(1 + \int \limits_{\R^{6}}|v|^{\frac{2p + \gamma}{2-\gamma}}d\mu_0(r,v)\right)\left(1 +  t^{\frac{2p+2}{2-\gamma}}\right), t \geq 0,
 \end{align}
 provided the right-hand sides are finite. Further, given a finite $T > 0$, for $t \in [0,T]$,
 \begin{align}
     \label{gamma-0/2-estimate}
     E|V_t|^{\frac{8}{2-\gamma}} < M_{\gamma,t}
 \end{align}
 for any $t \geq 0$, where $M_{\gamma,t}$ is a finite constant dependent on time and $\gamma$.
\end{thm}
When utilizing the above estimates, the symbol $C_t$ or $K_t$ is used to denote a constant in the estimate which depends on time. While there is an assumption of initial higher order estimates, the conservation laws still give additional, useful information. Conservation of energy yields a bound $E|V_t|^p \leq 1 + E|V_0|^2$ for all $0 < p < 2$ that is independent of time.
However, our results are not hindered by estimates that depend on time. We simply note that some moments can be bounded independent of time.
\section{Approximating Process}
For $\varepsilon > 0$, we construct an additive process $(R_t^\varepsilon,V_t^\varepsilon)$ that closely approximates the Boltzmann-Enskog process $(R_t,V_t)$, converging to it in mean as $\varepsilon \to 0$. The motivation being to obtain estimates leading to (\ref{deb-rom-eqn}) for the distribution of the Boltzmann-Enskog process though use of the properties of additive processes and their characteristic functions. Recall $\chi$ is a probability space with measure $\eta$ and $\Xi = [0,\pi) \times S^{1}$. We denote
\begin{equation*}
    \hat{\alpha}(v,r,u,q,\theta,\xi,z) = \alpha(v,u,\theta,\xi)1_{[0,\sigma(v - u)\beta(r - q)]}(z)
\end{equation*}
where $\alpha$ is defined in (\ref{alpha-defn}) and $\xi_0 = \xi_0(V_s -u_s,V_{t-\varepsilon}-u_s,\xi)$ is introduced in Lemma \ref{xi-lemma}.
\begin{prop}
\label{approx-def}
Let $(R_t,V_t)$ be an Boltzmann-Enskog process, satisfying the hypotheses of Theorem \ref{sund-estimate-thm}. Define an approximating process $(R_t^\varepsilon,V_t^\varepsilon)$
\begin{equation*}
    R_t^\varepsilon = R_{t-\varepsilon} + \int_{t - \varepsilon}^tV_s^\varepsilon\,ds
\end{equation*}
\begin{equation*}
    V_t^\varepsilon = V_{t-\varepsilon} + \int_{t-\varepsilon}^t\int_{\Xi\times\chi\times\mathbb{R}_+}\hat{\alpha}(V_{t-\varepsilon},R_{t-\varepsilon},u_s(\eta),q_s(\eta),\theta,\xi + \xi_0,z)dN(s,\theta,\xi,\eta,z)
\end{equation*}
Then for finite $t$, $(R_t^\varepsilon,V_t^\varepsilon) \to (R_t,V_t)$ in $L^1(P)$ as $\varepsilon \to 0$. Further, for some positive constant $K_t$ dependent on $t$ and strictly positive power $\kappa_\gamma > 1$,
\begin{align}
    \label{kappa-bound}
    E|V_t - V^\varepsilon_t| \leq K_t\varepsilon^{\kappa_\gamma}
\end{align}
\end{prop}
\begin{proof}\,\\
    \textit{Step 1.} First, Lipschitz continuity under expectation with respect to time for the stochastic process $V$ is established. Namely, given $0 \leq s < t$,
    \begin{equation}
    \label{lipschitz-in-time}
        \begin{split}
        E|V_t - V_s| \leq K_{t}(t - s)
       \end{split}
    \end{equation}
    for some constant $K_t > 0$ that remains finite as $s \to t$. We observe $K_t$ can be made independent of time provided $\gamma \in (0,1]$. \\ For $\gamma \in (0,2)$, examine
    \begin{equation*}
        \begin{split}
            E|V_t - V_u| & \leq E\left[\int_u^t\int_{\Xi\times\chi\times\mathbb{R}_+}|\hat{\alpha}(V_{s-},R_{s},u_s(\eta),q_s(\eta),\theta,\xi,z)|dN(s,\theta,\xi,\eta,z)\right]\\
            & = E\left[\int_u^t\int_{\Xi\times\chi\times\mathbb{R}_+}|\hat{\alpha}(V_{s-},R_{s},u_s,q_s,\theta,\xi,z)|dzQ(d\theta)d\xi d\eta ds\right]\\
            & = E\left[\int_u^t\int_{\Xi\times\chi}|\alpha(V_{s-},u_s,\theta,\xi)|\sigma(V_{s-} - u_s)\beta(R_s - q_s)Q(d\theta)d\xi d\eta ds\right]\\
            & \leq E\left[\int_u^t\int_{\Xi\times\chi}C\theta(|V_{s-} - u_s| + |V_{s-} - u_s|^{1+\gamma})\beta(R_s - q_s)Q(d\theta)d\xi d\eta ds\right]\\
            & \leq CE\left[\int_u^t\int_{\chi}|V_{s-} - u_s| + |V_{s-} - u_s|^{1 + \gamma} d\eta ds\right]\\
            \end{split}
    \end{equation*}
    using (\ref{alpha-lip}), (\ref{sigma-gamma}), bounding $\beta$, and the fact that $\int_0^\pi\theta Q(d\theta) < \infty$. 
    \begin{equation*}
        \begin{split}
            &CE\left[\int_u^t\int_{\chi}|V_{s} - u_s| + |V_{s} - u_s|^{1 + \gamma} d\eta ds\right]\\
            & = C\int_u^t\int_{\chi}E[|V_{s} - u_s| + |V_{s} - u_s|^{1 + \gamma}] d\eta ds\\
            & \leq C\int_u^t\int_{\chi}[E|V_{s}|^{1+\gamma} + E|u_s|^{1 + \gamma} + E|V_{s}| + E|u_s|] d\eta ds\\
            & \leq C_{t}(t - u)
        \end{split}
    \end{equation*}
    where we use $E|V_{s} - u_s|^{1 + \gamma} \leq 2^{\gamma}(E|V_{s}|^{1+\gamma} + E|u_s|^{1 + \gamma})$. The finite values of these moments gives the final bound. For $0 \leq u < t$,
    \begin{equation*}
        E|R_t - R_u| \leq \int_u^tE|V_s|\,ds \leq M(t - u)
    \end{equation*}
    Next we show that
    \begin{equation}
    \label{vt - vu ineq}
        E|V_t - V_u|^{1 + \gamma} \leq K(t - u)
    \end{equation}
    for some positive $K$. In the following lines, we will use the abbreviation $\hat{a}$ to shorten $\hat{\alpha}(V_s,R_s,u_s,q_s,\theta,\xi,z)$ for the sake of legibility. Applying the It\^{o} formula to $|V_t - V_u|^{1 + \gamma}$ yields
    \begin{equation*}
        |V_t - V_u|^{1 + \gamma} = \int_u^t\int_{\chi\times\Xi\times\mathbb{R}_+}(|V_s + \hat{a}|^{1 + \gamma} - |V_s|^{1 + \gamma})dN
    \end{equation*}
    In some calculations, we will utilize the following inequalities. For $a,b >0$, there exists constants $0 < c_{a,b} < C_{a,b}$ such that for all $x,y > 0$:
    \begin{equation}
    \label{ab-ineq}
        c_{a,b}|x^{a+b} - y^{a+b}| \leq (x^{a} + y^{a})|x^b - y^b| \leq C_{a,b}|x^{a+b}-y^{a+b}|
    \end{equation}
    Using it for the specific case $a = 1 + \gamma$ and $b = 1$, we obtain 
    \begin{align*}
        & E\left[\int_u^t\int_{\chi\times\Xi\times\mathbb{R}_+}|V_s + \hat{\alpha}|^{1 + \gamma} - |V_s|^{1 + \gamma}dN\right] \leq K_{\gamma}E\left[\int_u^t\int_{\chi\times\Xi\times\mathbb{R}_+}(|V_s|^\gamma + |\hat{\alpha}|^\gamma)|\hat{\alpha}|d\hat{N}\right]\\
        & = \int_u^t\int_{\chi\times\Xi\times\mathbb{R}_+}K_\gamma E[(|V_s|^\gamma + |\hat{\alpha}|^\gamma)|\hat{\alpha}|]d\hat{N} 
    \end{align*}
    Next,
    \begin{equation*}
        \begin{split}
            & \int_u^t\int_{\chi\times\Xi\times\mathbb{R}_+}K_{\gamma}E[(|V_s|^\gamma + |\hat{\alpha}|^\gamma)|\hat{\alpha}|]d\hat{N}\\
            & \leq \int_u^t\int_{\chi\times\Xi}K_{\gamma}\theta^{1 + \gamma} E[(|V_s|^\gamma + |V_s - u_s|^{\gamma})(|V_s - u_s| + |V_s - u_s|^{1 + \gamma})]Q(d\theta)d\xi d\eta ds\\
            & \leq \int_u^t\int_{\chi}K_{\gamma} \left(E[|V_s|^{1 + 2\gamma}] + E[|V_s|^\gamma|u_s|
            ^{1 + \gamma}] + E[|V_s - u_s|^{1 + 2\gamma}] + E[|V_s - u_s|^{1 + \gamma}]\right.\\
            & + \left. E|V_s|^{1 + \gamma} + E[|V_s|^\gamma|u_s|]\right) d\eta ds\\
            & \leq \int_u^t\int_{\chi}K_{\gamma} \left(E[|V_s|^{1 + 2\gamma}] + E[|V_s|^{2\gamma}]^{1/2}E[|u_s|
            ^{2 + 2\gamma}]^{1/2} + E[|V_s|^{1 + 2\gamma}] +E[|u_s|^{1 + 2\gamma}]\right.\\
            & \left. + E|V_s|^{1 + \gamma} + E|u_s|^{1 + \gamma} + E|V_s|^{1 + \gamma} + E[|V_s|^{2\gamma}]^{1/2}E[|u_s|^2]^{1/2}\right)d\eta ds
        \end{split}
    \end{equation*}
    using (\ref{alpha-lip}) and (\ref{sigma-gamma}). We then integrate with respect to $\xi$ and $\theta$, noting that\\ $ \int_0^\pi \theta^{1 + \gamma} Q(d\theta) < \infty$. We claim each term in the integrand is finite. Recall by (\ref{gamma-0/2-estimate}), for $\gamma \in (0,2)$,
    \begin{equation*}
        E[|V_t|^{\frac{8}{2 - \gamma}}] < M_{t,\gamma}
    \end{equation*}
    for each $t \geq 0$. Similarly, $\int_\chi|u_t|^{\frac{8}{2 - \gamma}}d\eta < \infty$. Notice that each moment above is less than $\frac{8}{2 - \gamma}$ for $\gamma \in (0,2)$. Therefore,
    \begin{equation*}
    \begin{split}
        & \int_{\chi}K \left(E[|V_s|^{1 + 2\gamma}] + E[|V_s|^{2\gamma}]^{1/2}E[|u_s|
            ^{2 + 2\gamma}]^{1/2} +E[|u_s|^{2 + 2\gamma}]\right.\\
            & \left. + E|V_s|^{1 + \gamma} + E|u_s|^{1 + \gamma} +  E[|V_s|^{2\gamma}]^{1/2}E[|u_s|^2]^{1/2}\right)d\eta \leq M_{s,\gamma}
    \end{split}
    \end{equation*}
    for some constant $M_{s,\gamma}$ dependent on $s$ and $\gamma$. Thus,
    \begin{equation}
    \label{vt-vu 1 + gamma}
        E[|V_t - V_u|^{1 + \gamma}] \leq K_{t,\gamma}(t - u)
    \end{equation}
    for some constant $K_{t,\gamma}$ that is finite for finite $t$ and $\gamma \in (0,2)$.\\
    Take $c_\gamma \in (1,8/(2 - \gamma))$. Then, $E|R_t - R_u|^{c_{\gamma}} \leq  K_{t}(t - u)$ for some finite $K_t$.
    Indeed,
    \begin{equation}
    \label{R-lip}
    \begin{split}
        E|R_t - R_u|^{c_\gamma} & \leq E\left[\left(\int_u^t|V_s|ds\right)^{c_\gamma}\right]\\
        & \leq E\left[(t - u)^{\frac{c_\gamma - 1}{c_\gamma}}\int_u^t|V_s|^{c_\gamma}ds\right]\\
        & = (t - u)^{\frac{c_\gamma - 1}{c_\gamma}}\int_u^tE|V_s|^{c_\gamma}ds \leq K_{t}(t - u)^\frac{2c_\gamma - 1}{c_\gamma}
    \end{split}
    \end{equation}
    using H\"{o}lder's Inequality with $p = c_\gamma$ and $q = \frac{c_\gamma}{c_\gamma - 1}$ and that $E|V_s|^{c_\gamma} < \infty$ for each $s \in [u,t]$.
    \textit{Step 2.} Now we establish $E|V_t - V_t^\varepsilon| \leq K_t \varepsilon^{\kappa}$ for $\varepsilon,K_t > 0$, some power $\kappa > 1$, and a finite time $t > 0$. By definition, 
    \begin{equation}
    \label{initial}
    \begin{split}
        &E|V_t - V_t^\varepsilon|\\ 
        & = E\left|\int_{t-\varepsilon}^{t}\int_{\Xi\times\chi\times\mathbb{R}_+}\hat{\alpha}(V_{s-},R_s,u_s,q_s,\theta,\xi,z) - \hat{\alpha}(V_{t-\varepsilon},R_{t-\varepsilon},u_s,q_s,\theta,\xi + \xi_0,z)dN\right|
        \end{split}
    \end{equation}
    Bringing the absolute value inside the integral then integrating in $z$ gives that (\ref{initial}) is less than or equal to
    \begin{equation*}
    \begin{split}
        & \int_{t-\varepsilon}^t\int_{\Xi\times\chi}E[|\alpha(V_s,u_s,\theta,\xi)\sigma(V_s - u_s)\beta(R_s - q_s)\\
        & - \alpha(V_{t-\varepsilon},u_s,\theta,\xi + \xi_0)\sigma(V_{t-\varepsilon} - u_s)\beta(R_{t-\varepsilon} - q_s)|Q(d\theta)d\xi d\eta ds]
    \end{split}
    \end{equation*}
    which is less than or equal to 
    \begin{equation*}
        \int_{t-\varepsilon}^tE[B_1(s) + B_2(s) + B_3(s)]ds
    \end{equation*}
    where 
    \begin{equation*}
    \begin{split}
        & B_1(s) = \int_{\Xi\times\chi}[\sigma(V_s - u_s)\beta(R_s - q_s) \wedge \sigma(V_{t - \varepsilon} - u_s)\beta(R_{t- \varepsilon} - q_s)]\\
        & \times |\alpha(V_s,u_s,\theta,\xi) - \alpha(V_{t-\varepsilon},u_s,\theta,\xi + \xi_0)|d\eta d\xi Q(d\theta)
    \end{split}
    \end{equation*}
    \begin{equation*}
    \begin{split}
        & B_2(s) = \int_{\Xi\times\chi}[\sigma(V_s - u_s)\beta(R_s - q_s) - \sigma(V_{t - \varepsilon} - u_s)\beta(R_{t- \varepsilon} - q_s)]_+\\
        & \times|\alpha(V_s,u_s,\theta,\xi)|d\eta d\xi  Q(d\theta)
    \end{split}
    \end{equation*}
    \begin{equation*}
    \begin{split}
        & B_3(s) = \int_{\Xi\times\chi}[\sigma(V_{t-\varepsilon}- u_s)\beta(R_{t - \varepsilon} - q_s) - \sigma(V_{s} - u_s)\beta(R_{s} - q_s)]_+\\
        & \times |\alpha(V_{t - \varepsilon},u_s,\theta,\xi + \xi_0)|d\eta d\xi  Q(d\theta)
    \end{split}
    \end{equation*}
    Using the boundedness of $\beta$, (\ref{sigma-lip}) and (\ref{alph-cont}), there is the following bound for $B_1(s):$
    \begin{equation*}
        \begin{split}
            B_1(s) & \leq \int_{\Xi\times\chi}[\sigma(V_s - u_s) \wedge \sigma(V_{t - \varepsilon} - u_s)]\\
        & \times |\alpha(V_s,u_s,\theta,\xi) - \alpha(V_{t-\varepsilon},u_s,\theta,\xi + \xi_0)|d\eta d\xi Q(d\theta)\\
        & \leq 2\int_{\Xi\times\chi}[\sigma(V_s - u_s) \wedge \sigma(V_{t - \varepsilon} - u_s)] \theta|V_s - V_{t-\varepsilon}| d\eta d\xi Q(d\theta)
        \end{split}
    \end{equation*}
    Noting $\int_0^\pi \theta Q(d\theta) < \infty$, then using (\ref{sigma-gamma}) establishes
    \begin{equation*}
        \begin{split}
            & 2\int_{\Xi\times\chi}[\sigma(V_s - u_s) \wedge \sigma(V_{t - \varepsilon} - u_s)] \theta|V_s - V_{t-\varepsilon}| d\eta d\xi Q(d\theta)\\
            & = 2C_1\int_{\chi}[\sigma(V_s - u_s) \wedge \sigma(V_{t - \varepsilon} - u_s)]|V_s - V_{t-\varepsilon}| d\eta\\
            & \leq 2^{\gamma + 1} C_1\int_{\chi}[(1 + |V_s - u_s|^\gamma) \wedge (1 + |V_{t-\varepsilon} - u_s|^\gamma)]|V_s - V_{t-\varepsilon}| d\eta\\
        \end{split}
    \end{equation*}
    where $C_1 := \int_\Xi \theta Q(d\theta) d\xi$. Call this final upper bound as $A_1^{\eta}(s)$. For $B_2(s)$ and $B_3(s)$, we use (\ref{alpha-lip}), then integrate in $\theta$ and $\xi$ to give
    \begin{equation*}
        B_2(s) \leq C_1\int_{\chi}[\sigma(V_s - u_s)\beta(R_s - q_s) - \sigma(V_{t - \varepsilon} - u_s)\beta(R_{t- \varepsilon} - q_s)]_+|V_s - u_s|d\eta
    \end{equation*} 
    \begin{equation*}
        B_3(s) \leq C_1\int_\chi [\sigma(V_{t-\varepsilon}- u_s)\beta(R_{t - \varepsilon} - q_s) - \sigma(V_{s} - u_s)\beta(R_{s} - q_s)]_+|V_{t-\varepsilon} - u_s|d\eta
    \end{equation*}
    and call these upper bounds $A_2^\eta(s)$ and $A_3^\eta(s)$ respectively. Thus,
    \begin{equation*}
        E|V_t - V_t^\varepsilon| \leq C\int_{t - \varepsilon}^tE[A_1^\eta(s) + A_2^\eta(s) + A_3^\eta(s)]ds
    \end{equation*}
    where $C = 2C_1$.\\
    \textit{Step 3.} We now show that the expectation of $A^\eta_i$ is bounded above by some constant dependent on $\varepsilon$. This is first established for the $\gamma \in (0,1]$ case. For some constant $K_1$,
    \begin{equation*}
        A_1^\eta(s) \leq K_1\int_\chi (1 + |u_s|^\gamma + |V_s|^\gamma + |V_{t - \varepsilon}|^\gamma)|V_s - V_{t - \varepsilon}|d\eta
    \end{equation*}
    Now examine $A^\eta_2 + A^\eta_3$. First note
    \begin{equation*}
        \begin{split}
        & A_2^\eta(s) + A_3^\eta(s) \leq K'\int_{\chi}|\sigma(V_s - u_s)\beta(R_s - q_s) - \sigma(V_{t - \varepsilon} - u_s)\beta(R_{t- \varepsilon} - q_s)|\\
        & \times [|V_s - u_s| + |V_{t - \varepsilon} - u_s|]d\eta
        \end{split}    
    \end{equation*}
    and then examine $|\sigma(V_s - u_s)\beta(R_s - q_s) - \sigma(V_{t - \varepsilon} - u_s)\beta(R_{t- \varepsilon} - q_s)|$. Initially,
    \begin{equation*}
        \begin{split}
            & |\sigma(V_s - u_s)\beta(R_s - q_s) - \sigma(V_{t - \varepsilon} - u_s)\beta(R_{t- \varepsilon} - q_s)|\\
            & \leq \beta(R_s - q_s)|\sigma(V_s - u_s) - \sigma(V_{t - \varepsilon} - u_s)|\\
            & + \sigma(V_{t - \varepsilon} - u_s)|\beta(R_s - q_s) - \beta(R_{t - \varepsilon} - q_s)|
        \end{split}
    \end{equation*}
    Using the boundedness of $\beta$ and (\ref{sigma-lip}), 
    \begin{equation}
    \label{sigmaBound}
        \beta(R_s - q_s)|\sigma(V_s - u_s) - \sigma(V_{t - \varepsilon} - u_s)| \leq K''||V_s - u_s|^\gamma - |V_{t-\varepsilon} - u_s|^\gamma|
    \end{equation}
    for some constant $K''$. Since $\beta \in C^1_c(\mathbb{R}^3)$ and non-negative, it is Lipschitz as well. Using this and property (\ref{sigma-gamma}),
    \begin{equation*}
        \sigma(V_{t - \varepsilon} - u_s)|\beta(R_s - q_s) - \beta(R_{t - \varepsilon} - q_s)| \leq K'''(1 + |V_{t - \varepsilon} - u_s|^\gamma)|R_s - R_{t- \varepsilon}|
    \end{equation*}
    for some constant $K'''$. Thus, given a large enough constant $K$,
    \begin{equation}
    \label{a_2+a_3 bound}
        \begin{split}
            & A_2^\eta(s) + A_3^\eta(s)\\
            & \leq K\int_\chi(|R_s - R_{t - \varepsilon}| + |V_s - V_{t-\varepsilon}|^\gamma +|V_{t-\varepsilon} - u_s|^\gamma||R_s - R_{t - \varepsilon}|)[|V_s - u_s| + |V_{t - \varepsilon} - u_s|]d\eta\\
        \end{split}
    \end{equation}
    thanks to the $\gamma$ being within $(0,1]$. We now bound the expectations of $A_1,A_2,A_3$. For $A_1$, we have 
    \begin{equation*}
    \begin{split}
        E[A_1^\eta(s)] & \leq K_1\int_\chi E[(1 + |u_s|^\gamma + |V_s|^\gamma + |V_{t-\varepsilon}|^\gamma)|V_s - V_{t-\varepsilon}|] d\eta\\
        & \leq K_1'\int_\chi E[1 + |u_s|^{1 + \gamma} + |V_s|^{1 + \gamma} + |V_{t - \varepsilon}|^{1 + \gamma}]^{\gamma/(\gamma + 1)}E[|V_s - V_{t-\varepsilon}|^{1 + \gamma}]^{1/(1 + \gamma)}d\eta\\
        & \leq K\varepsilon^{1/(1+\gamma)}
    \end{split}
    \end{equation*}
    using a H\"{o}lder inequality with $p = 1 + \gamma$ and $q = 1 + \frac{1}{\gamma}$, allow a constant $K_1'$ to split the new power to each term in the first expectation, then the finite moments of the processes and (\ref{vt - vu ineq}). Now for the sum of $A_2$ and $A_3$, similar to $A_1$ (using the same $p$ and $q$):
    \begin{equation}
    \label{bound0}
    \begin{split}
        &E[|V_s -  V_{t - \varepsilon}|^\gamma|V_s - u_s|] \leq E[|V_s - V_{t-\varepsilon}|^{1 + \gamma}]^{\gamma/(\gamma + 1)}E[|V_s - u_s|^{1 + \gamma}]^{1/(\gamma + 1)}\\
        & \leq KE[|V_s - V_{t-\varepsilon}|^{1 + \gamma}]^{\gamma/(\gamma + 1)}E[|V_s|^{1 + \gamma} + |u_s|^{1 + \gamma}]^{1/(\gamma + 1)} \leq K_t\varepsilon^{\gamma/(\gamma + 1)}
    \end{split}
    \end{equation}
    Likewise $E[|V_s - V_{t - \varepsilon}|^{\gamma}|V_{t-\varepsilon} - u_s|] \leq K\varepsilon^{\gamma/(\gamma + 1)}$.\\  Next,
    \begin{equation}
    \label{bound1}
        E[|V_{t-\varepsilon} - u_s|^{1 + \gamma}|R_s - R_{t-\varepsilon}] \leq E[|V_{t-\varepsilon} - u_s|^{2 + 2\gamma}]^{1/2}E[|R_s - R_{t-\varepsilon}|^2]^{1/2} \leq K_{t,\gamma}\varepsilon^{3/4}
    \end{equation}
    as $E|R_s - R_{t-\varepsilon}|^2 \leq \varepsilon^{3/2}$. Likewise, 
    \begin{equation}
    \label{bound2}
        \begin{split}
            & E[|V_{t-\varepsilon} -u_s|^\gamma|V_s - u_s||R_s - R_{t-\varepsilon}|] \\
            & \leq E[|V_{t - \varepsilon} - u_s|^{2\gamma}|V_s - u_s|^2]^{1/2}E[|R_s - R_{t - \varepsilon}|^{2}]^{1/2}\\
            & \leq E[|V_{t - \varepsilon} - u_s|^{2\gamma}]E[|V_s - u_s|^4]^{1/4}E[|R_s - R_{t - \varepsilon}|^{2}]^{1/2}\\
            & \leq K_{s,\gamma}\varepsilon^{3/4}
        \end{split}
    \end{equation}
    using the Cauchy-Schwarz inequality and the bound on $E|R_s - R_{t - \varepsilon}|^{2}.$\\
    Using previous bounds, 
    \begin{equation}
    \label{bound3}
        \begin{split}
            &E[|V_s -  u_s||R_s - R_{t-\varepsilon}|] \leq E[|V_s - u_s|^{2}]^{1/2}E[|R_s - R_{t-\varepsilon}|^{2}]^{1/2}\\
        & \leq K_{s,\gamma}\varepsilon^{3/4}
        \end{split}
    \end{equation}
    and likewise $E[|V_{t-\varepsilon} -  u_s||R_s - R_{t-\varepsilon}|] \leq K\varepsilon^{3/2}$.\\
    Collecting bounds (\ref{bound0})-(\ref{bound3}), there is a large enough constant $K_{t,\gamma}$ and a positive power $\kappa$ such that $E|V_t - V_t^\varepsilon| \leq K_{t,\gamma}\varepsilon^{\kappa}$. In fact, $\kappa = 1 + \frac{\gamma}{\gamma + 1}$ as it is the minimal power available. For position,
    \begin{equation*}
        E|R_t - R_t^\varepsilon| \leq \int_{t - \varepsilon}^tE|V_s - V_{s}^\varepsilon|ds \leq K_{t,\gamma}\varepsilon^{\kappa + 1}
    \end{equation*}
    and as such, $(R_t^\varepsilon,V_t^\varepsilon) \to (R_t,V_t)$ in $L^1$ as $\varepsilon \to 0$ for $\gamma \in (0,1]$.
    \textit{Step 4.} Examine now $\gamma \in (1,2)$. The first two steps of the proof still apply when $\gamma \in (1,2)$. However, using (\ref{sigmaBound}) to get to (\ref{a_2+a_3 bound}) will not work as $|x|^k$ is not H\"{o}lder continuous for $k > 1$. Instead we use (\ref{ab-ineq}) to give
    \begin{equation*}
        \begin{split}
            &K''||V_s - u_s|^\gamma - |V_{t-\varepsilon} - u_s|^\gamma|\\
            &\leq K''C_{\gamma}||V_s - u_s| - |V_{t-\varepsilon} - u_s||[|V_s - u_s|^{\gamma - 1} + |V_{t-\varepsilon} - u_s|^{\gamma - 1}]\\
            & \leq K''C_\gamma|V_s - V_{t-\varepsilon}|[|V_s - u_s|^{\gamma - 1} + |V_{t-\varepsilon} - u_s|^{\gamma - 1}]
        \end{split}
    \end{equation*}
    where this is allowed as $\gamma > 1$ so that $\gamma - 1 > 0$ and thus we can use a standard inequality given below.
 Thus, there is some large enough constant to say
    \begin{equation*}
        \begin{split}
            & A_2^\eta(s) + A_3^\eta(s)\\
            & \leq K\int_\chi(|R_s - R_{t - \varepsilon}| + |V_s - V_{t-\varepsilon}||V_s - u_s|^{\gamma - 1} \\
            &\,\,\, + |V_s - V_{t - \varepsilon}||V_{t-\varepsilon} - u_s|^{\gamma - 1} +|V_{t-\varepsilon} - u_s|^\gamma||R_s - R_{t - \varepsilon}|)[|V_s - u_s| + |V_{t - \varepsilon} - u_s|]d\eta\\
        \end{split}
    \end{equation*}
    Many of the bounds above still work, and it still holds that $EA_1^\eta(s) \leq K\varepsilon^{1/(\gamma+1)}$, as in Step 3. However, there are new bounds to establish for the sum $A_2^\eta(s) + A_3^\eta(s)$. First, by H\"{o}lder's inequality,
    \begin{equation*}
        E[|V_s - V_{t-\varepsilon}||V_s - u_s|^\gamma] \leq E[|V_s - V_{t-\varepsilon}|^{1 + \gamma}]^{1/(\gamma + 1)}E[|V_s - u_s|^{\gamma + 1}]^{\gamma/(\gamma + 1)} \leq K\varepsilon^{1/(\gamma + 1)}.
    \end{equation*}
    Similarly, $E[|V_s - V_{t-\varepsilon}||V_{t-\varepsilon} - u_s|^{\gamma}] \leq K\varepsilon^{\gamma/(\gamma + 1)}$. Now, we examine $E[|V_s - V_{t-\varepsilon}||V_s - u_s|^{\gamma - 1}|V_{t-\varepsilon} - u_s|]$. We use the H\"{o}lder inequality to establish,
    \begin{equation*}
        \begin{split}
            & E[|V_s - V_{t-\varepsilon}||V_s - u_s|^{\gamma - 1}|V_{t-\varepsilon} - u_s|]\\
            & \leq E[|V_s - V_{t-\varepsilon}|^{\gamma + 1}]^{1/(\gamma + 1)}E[|V_s - u_s|^{\gamma - \frac{1}{\gamma}}|V_{t-\varepsilon} - u_s|^{1 + \frac{1}{\gamma}}]^{\gamma/(\gamma + 1)}\\
            & \leq K\varepsilon^{1/(\gamma + 1)}E[|V_s - u_s|^{2\gamma - \frac{2}{\gamma}}]^{\gamma/(2\gamma + 2)}E[|V_{t-\varepsilon} - u_s|^{2 + \frac{2}{\gamma}}]]^{\gamma/(2\gamma + 2)}\\
            & \leq K_{t,\gamma}\varepsilon^{1/(\gamma + 1)}
        \end{split}
    \end{equation*}
    where the latter comes from the fact that $\frac{8}{2 - \gamma} \geq 2\gamma + 2$ and $\frac{8}{2 - \gamma} \geq 2 + \frac{2}{\gamma}$ for $\gamma \in (1,2)$. Thus, we may use (\ref{vt - vu ineq}). Likewise, $E[|V_s - V_{t-\varepsilon}||V_{t - \varepsilon} - u_s|^{\gamma - 1}|V_s - u_s|] \leq K\varepsilon^{1/(\gamma + 1)}$ using a similar argument. The remainder of Step 4 carries over from the $\gamma \in (0,1]$ case.\\ In total, it can be said that $E|V_t - V_t^\varepsilon| \leq K_{t,\gamma} \varepsilon^{\kappa_\gamma}$ where $\kappa_\gamma = 1 + \frac{1}{\gamma + 1} \wedge 1 + \frac{\gamma}{\gamma + 1}$, depending on the specific $\gamma$ and $K_{t,\gamma}$ is a finite constant dependent on $t$ to a positive power and $\gamma$. 
\end{proof}
The convergence established above leads us to the idea that Theorem \ref{deb-rom-exist} may be applicable to the Boltzmann-Enskog process. However, a more powerful estimate is needed to use Theorem \ref{deb-rom-exist}.
\\ \textbf{A Useful Estimate for the Main Theorem}\\
The bounds in Theorem \ref{approx-def} are insufficient for satisfying the hypotheses of Theorem \ref{deb-rom-exist}. A fruitful estimate is given in the following lemma.
\begin{lemma}
\label{approx-est-lem}
Make the assumptions on $\alpha$, $\beta$ and $\sigma$ as stated in section \ref{assume} for some $\gamma \in (0,2) \text{ and } \nu \in (0,1)$. Let
$f_0$ be an $\R^6$ probability measure such that it is not a Dirac mass. Assume further $\int_{\R^6}|v|^2f_0(dr,dv) < \infty$. Consider the approximating process $(R^\varepsilon,V^\varepsilon)$, defined in Theorem \ref{approx-def} associated with a weak solution $f_t$ to (\ref{Enskog-weak}) starting from
$f_0$. For all $h \in \R^3$ , all $\phi \in L^\infty(\R^3)$, all $0 < t_0 \leq t- \varepsilon \leq t_1$ with $\varepsilon \in (0,1)$,
\end{lemma}
\begin{align}
    |E[\phi(V_t^\varepsilon + h)) - \phi(V_t^\varepsilon)]| \leq C_{t_0,t_1}\norm{\phi}_{L^\infty(\mathbb{R}^3)}|h|\varepsilon^{-1/\nu}
\end{align}
This result is carefully built in the sections that follow. In order to establish this lemma, more information on the support of the Boltzmann-Enskog distribution is required. The subsequent section of this paper will center around the measure-theoretic properties of the probability distribution of the Boltzmann-Enskog process.
\section{Establishing Support on $\mathbb{R}^3$}
This section is dedicated to establishing that the support of the marginal distribution for velocity has support in $\R^3$. In \cite{four}, Fournier provided an intuitive result that details a sufficient condition for the support of an $\mathbb{R}^3$ measure being all of $\mathbb{R}^3$. This result is given below.
\begin{lemma}[Lemma 4.1 from Fournier\,\cite{four}]
\label{fournier-supp}
Consider $g \in \mathcal{P}(\mathbb{R}^3)$ enjoying the following property: if $v_1,v_2 \in \text{Supp}(g)$, then $\mathcal{S}((v_1 + v_2)/2,|v_1 - v_2|/2) \subset \text{Supp}(g)$. If $g$ is not a Dirac mass, then $\text{Supp}(g) = \mathbb{R}^3$.     
\end{lemma}
The proof is entirely geometrical and topological, with the symmetries applied being universal to $\mathbb{R}^3$ vectors. To reduce ambiguity, we recall the definition of support for Radon measures. We recall that a Radon measure on $X$ is a Borel measure such that it is finite on all compact sets, outer regular on all Borel sets and inner regular on all open sets. 
\begin{defn}
    Let $\mu$ be a Radon measure on some space $X$. Let $N$ be the union of all open $U \subset X$ such that $\mu(U) = 0$. The complement of $N$ is called the \textbf{support} of $U$.
\end{defn}
Then clearly, if $X$ is a locally compact metrizable Hausdorff space with Radon measure $\mu$, $x \in \text{Supp}(\mu)$ if and only if for every $\varepsilon > 0$, $\mu(B(x,\varepsilon)) > 0$. We note that the distribution of the Boltzmann-Enskog process is Radon. Indeed, by Theorem 7.8 in \cite{foll}, 
\begin{thm}
Let $X$ be a locally compact Hausdorff space in which every open set is $\sigma$-compact (for example, if $X$ is second countable). Then every Borel measure on $X$ that is finite on compact sets is regular and hence Radon.
\end{thm}
Therefore finite measures on $\mathbb{R}^n$ are Radon, and we may use the notion of support defined above. 
\begin{prop}
Suppose that $(V_t,R_t)$ is an Boltzmann-Enskog process with $\gamma \in (0,2)$ and initial distribution $\mu_0$ which is not a Dirac mass. Then the support of the marginal distribution of $\mu_t$ with respect to $v$ is $\mathbb{R}^3$. Namely, $m_t(dv) = \int_{\mathbb{R}^3}\mu_t(dx,\circ)$ has support in the entirety of $\mathbb{R}^3$.
\end{prop}
\begin{proof}
We follow the proof of Fournier, accounting for changes in position-based terms. 
\begin{itemize}
    \item [Step 1.] First, we note that 
    \begin{equation*}
        \int_{\mathbb{R}^6}|v - v_0|^2\mu_t(dx,dv) = \int_{\mathbb{R}^6}|v - v_0|^2\mu_0(dx,dv)
    \end{equation*}
    for any fixed $v_0 \in \mathbb{R}^3$ under the conservation of momentum and energy. Further,
    \begin{equation*}
        \int_{\mathbb{R}^6}|v - v_0|^2\mu_0(dx,dv) = \int_{\mathbb{R}^3}|v - v_0|^2m_0(dv) > 0
    \end{equation*}
   since $\mu_0$ is not a Dirac measure. Thus, 
    \begin{equation*}
        \int_{\mathbb{R}^3}|v - v_0|^2m_t(dv) > 0
    \end{equation*}
    for any fixed $v_0 \in \mathbb{R}^3$, meaning $m(t,dv)$ is not a Dirac measure for any time $t$.
    \item [Step 2.] Now we attempt to show that if for a point $v_0 \in \mathbb{R}^3$ and $\varepsilon > 0$, we have $m(t,B(v_0,\varepsilon)) = 0$, then we also have 
    \begin{equation*}
        \int_{\mathbb{R}^6}\int_{\mathbb{R}^3}\int_{\Xi}1_{\{v + \alpha(v,u,\theta,\xi) \in B(v_0,\varepsilon)\}}1_{\{v \neq u,\theta \in (0,\pi/2)\}}d\sigma m_t(du)m_t(dv) = 0
    \end{equation*}
    Let $\varepsilon > 0$ and $v_0 \in \mathbb{R}^3$ be given. Assume that $m_t(B(v_0,\varepsilon)) = 0$. Take $\phi_{\varepsilon,v_0} \in C^1_b(\mathbb{R}^3)$ to be strictly positive in $B(v_0,\varepsilon)$ and vanish outside of it. By Remark 2.7 in \cite{sund} (with the proof given in Theorem 4.3), we can note that the mapping 
    \begin{equation}
    \label{cty-integral}
        s \to \int_{\mathbb{R}^6}\phi_{v_0,\varepsilon}(v)\mu_s
    \end{equation}
    is continuous. For $\psi \in C_b^1(\mathbb{R}^6)$, they establish the continuity of
    \begin{equation}
    \label{cty-sund}
    s \to \int_{\mathbb{R}^{12}\times \Xi}(\mathcal{A}\psi)(r,v;q,u)\,\mu_s(dr,dv)\mu_s(dq,du)    
    \end{equation}
    where $\mathcal{A}$ is defined as in (\ref{A-operator}). Recall by (\ref{Enskog-weak}) 
    \begin{equation}
    \label{cty-expanded}
      \int_{\mathbb{R}^6}\psi(v)\,\mu_t(dr,dv) = \int_{\mathbb{R}^6}\psi(v)\,\mu_0(dr,dv) + \int_0^t\int_{\mathbb{R}^{12}}(\mathcal{A}\psi)(r,v;q,u)\mu_s(dr,dv)\mu_s(dq,du)
    \end{equation}
    then the continuity shown for (\ref{cty-sund}) gives the continuity of (\ref{cty-integral}). Differentiability of \eqref{cty-integral} is immediate using the continuity of (\ref{cty-sund}) in \eqref{cty-expanded}, taking $\phi_{v_0,\varepsilon}$ rather than $\psi$. Thus the mapping (\ref{cty-integral}) is continuous and differentiable. By hypothesis and the definition of $\phi_{v_0,\varepsilon}$, the mapping (\ref{cty-integral}) is non-negative and vanishes at $t$. Thus, its derivative at $t$ does as well. Using the weak formulation of the Boltzmann-Enskog equation (\ref{Enskog-weak}), we have 
    \begin{equation*}
    \begin{split}
        \int_{\mathbb{R}^{12}}\!\int_{0}^\pi\!\int_{S^2}\!& \sigma(|v - u|)\beta(r - q)[\phi_{\varepsilon,v_0}(v + \alpha(v,u,\theta,\xi)) - \phi_{\varepsilon,v_0}(v)]\\
        &\mu_t(dr,dv)\mu_t(dq,du)Q(d\theta)d\xi = 0.
    \end{split}
    \end{equation*}
    as since we took $\phi_{\varepsilon,v_0}$ to be solely dependent on velocity, we have $v \cdot \nabla_r \phi = 0$. However, we note that we immediately have 
    \begin{equation*}
        \int_{\mathbb{R}^{12}}\int_{0}^\pi\int_{S^2}\sigma(|v - u|)\beta(r - q)\phi_{\varepsilon,v_0}(v)\mu_t(dr,dv)\mu_t(dq,du)Q(d\theta)d\xi = 0
    \end{equation*}
    since (by assumption), 
    \begin{equation*}
        \begin{split}
        & \int_{\mathbb{R}^{12}}\int_{0}^\pi\int_{S^2}\sigma(|v - u|)\beta(r - q)\phi_{\varepsilon,v_0}(v)\mu_t(dr,dv)\mu_t(dq,du)Q(d\theta)d\xi\\
        & = \int_{\mathbb{R}^{6}}\int_{0}^\pi\int_{S^2}\sigma(|v - u|)\phi_{\varepsilon,v_0}(v)\int_{\mathbb{R}^6}\beta(r - q)\mu_t(dr,dv)\mu_t(dq,du)Q(d\theta)d\xi\\
        & \leq M\int_{\mathbb{R}^{6}}\int_{0}^\pi\int_{S^2}\sigma(|v - u|)\phi_{\varepsilon,v_0}(v)m_t(dv)m_t(du)Q(d\theta)d\xi\\
        & \leq M\int_{\mathbb{R}^{3}}\int_{0}^\pi\int_{S^2}\int_{\mathbb{R}^3}\sup_{v \in \mathbb{R}^3}\sigma(|v - u|)\phi_{\varepsilon,v_0}(v)m_t(B(v_0,\varepsilon))m_t(du)Q(d\theta)d\xi = 0
        \end{split}
    \end{equation*}
    where we use that $\phi_{v_0,\varepsilon}$ has support within $B(v_0,\varepsilon)$, so that for any fixed $u \in \mathbb{R}^3$, we have that $\sup_{v \in \mathbb{R}^3}\sigma(|v - u|)\phi_{v_0,\varepsilon}(v) < \infty$. Further, bounding $\beta$ above allows one to receive the velocity marginals. As such, we must have that 
    \begin{equation*}
        \int_{\mathbb{R}^6}\int_{0}^\pi\int_{S^2}\sigma(|v - u|)\phi_{\varepsilon,v_0}(v + \alpha(v,u,\theta,\xi))m_t(dv)m_t(du)Q(d\theta)d\xi = 0.
    \end{equation*}
    This implies our result, as $\sigma(|v - u|) \neq 0$ when $u \neq v$, and $\phi_{\varepsilon,0}(v + \alpha(v,u,\theta,\xi)) > 0$ when $v + \alpha(v,u,\theta,\xi) \in B(v_0,\varepsilon)$. 
    \item [Step 3.] We now show that for $t > 0$, if $v_1,v_2 \in \text{Supp}(m_t)$, then we have $\mathcal{S}((v_1 + v_2)/2,|v_1 - v_2|/2) \subset \text{Supp}(m_t)$. We can assume $v_1 \neq v_2$, otherwise $\mathcal{S}((v_1 + v_2)/2,|v_1 - v_2|/2) = \{v_1\}$ and the result is immediate. Observe that $\mathcal{S}((v_1 + v_2)/2,|v_1 - v_2|/2)$ is the closure of $\Delta_{v_1,v_2} \cup \Delta_{v_2,v_1}$ where 
    \begin{equation*}
        \Delta_{v_1,v_2} : = \{v_1 + \alpha(v_1,v_2,\theta,\xi):\xi \in S^2,\theta \in (0,\pi/2)\}
    \end{equation*}
    Since $\text{Supp}(m_t)$ is closed (as it is the complement of a union of open balls of measure zero), it suffices to prove that $\Delta_{v_1,v_2} \cup \Delta_{v_2,v_1} \subset \text{Supp}(m_t)$ as this gives the result by Lemma \ref{fournier-supp}. Take $v_0 \in \Delta_{v_1,v_2}$. Then for $v_0 = v_1 + \alpha(v_1,v_2,\theta_0,\xi_0)$ for some $\theta_0 \in (0,\pi/2)$ and $\xi_0 \in \mathcal{S}^2$. We will now show
    \begin{equation*}
        \int_{\mathbb{R}^3}\int_{\mathbb{R}^3}\int_{\Xi}1_{\{v + \alpha(v,u,\theta,\xi) \in B(v_0,\varepsilon)\}}1_{\{v \neq u,\theta \in (0,\pi/2)\}}d\sigma m_t(du)m_t(dv) > 0.
    \end{equation*}
    Let $\varepsilon > 0$ be given. We examine $B(v_0,\varepsilon)$. For $v$,$u$ and $\xi$ such that we keep them sufficiently close to $v_1$,$v_2$, and $\xi_0$ respectively, we have $v + \alpha(v,u,\theta,\xi) \in B(v_0,\varepsilon)$. As such, we can state that for $\delta > 0$ small, we have the set containment
    \begin{equation*}
        \begin{split}
        &\{v:B(v_1,\delta)\} \times \{u:B(v_2,\delta)\} \times \{\theta \in (a,b): 0 < Q((a,b)) < \infty\} \times \{\xi:B(\xi_0,\delta)\}\\
        & \subset \{v,u,\theta,\xi: v + \alpha(v,u,\theta,\xi) \in B(v_0,\varepsilon)\} \times \{v,u,\theta: v \neq u, \theta \in (0,\pi/2)\}
        \end{split}
    \end{equation*}
    Thus 
    \begin{equation*}
    \begin{split}
        & \int_{\mathbb{R}^3}\int_{\mathbb{R}^3}\int_{\Xi}1_{\{v + \alpha(v,u,\theta,\xi) \in B(v_0,\varepsilon)\}}1_{\{v \neq u,\theta \in (0,\pi/2)\}}d\xi Q(d\theta) m_t(du)m_t(dv)\\
        & \geq \int_{B(v_1,\delta)}\int_{B(v_2,\delta)}\int_{(a,b) \times B(\xi_0,\delta)}d\xi Q(d\theta) m_t(du)m_t(dv) > 0
    \end{split}
    \end{equation*}
    where this lower integral is positive using our corollary before since $v_1$ and $v_2$ are in the support for $t > 0$.   This implies that $m_t(B(v_0,\varepsilon)) > 0$ for all $\varepsilon > 0$ using the contrapositive of step 2, and thus $v_0$ is in the support by our corollary.
    \item [Step 4.] Using Lemma \ref{fournier-supp} and the steps above, we conclude that $\text{Supp}(m_t) = \mathbb{R}^3$ for all $t > 0$.
\end{itemize}
\end{proof}
\noindent \textbf{Distribution Estimates}\\
The support for the marginal distribution is useful for establishing estimates related to the velocity process $V_t$. In attempting to determine estimates on this process, specific geometric properties for arguments become invaluable. Lower bounds on the distribution are needed, and under constraints of keeping velocities separated by a certain degree, they can be obtained. As such, the following result is needed. 
\begin{prop}
\label{geometric prop}
Make the assumptions on $\alpha$,$\beta$ and $\sigma$ as stated in section \ref{assume} for $\gamma \in (0,2)$ and $\nu \in (0,1)$. Set $\mu_t$ to be the solution to the Boltzmann-Enskog equation with $\mu_t = \mathcal{L}(R_t,V_t)$ as defined in Theorem \ref{sund-estimate-thm}. For all $0 < t_0 < t_1$, 
\begin{equation*}
    c_{t_0,t_1} = \inf_{t \in [t_0,t_1],u\in \mathbb{R}^3,\xi \in \mathbb{R}^3}m_t(K(u,\xi)) > 0
\end{equation*}
where
\begin{equation*}
\begin{split}
    K(u,\xi)&  = \{v\in\mathbb{R}^3:|v| \leq 3,
    |v - u| \geq 1,\sigma(|v - u|) \geq m > 0,\\
    & |<v-u,\xi>| \geq |\xi|\}
\end{split}
\end{equation*}
and $m_t$ denotes the marginal distribution of $\mu_t$ for $v$.
\end{prop}
\begin{proof}\,
\begin{itemize}
       \item [Step 1.] We first prove that for $0 < t_0 < t_1$, we have $\inf_{t \in [t_0,t_1],x \in \mathcal{S}(0,2)}\mu_t(\mathcal{B}(x,1) \times \mathbb{R}^3) = \inf_{t \in [t_0,t_1],x \in \mathcal{S}(0,2)}m_t(\mathcal{B}(x,1)) > 0$ where $m_t$ denotes the marginal distribution of $\mu_t$ with respect to $v$. Consider then, $\phi \in \text{Lip}_b(\mathbb{R}^3)$ such that $1_{\{\mathcal{B}(0,\frac{1}{2})\}} \leq \phi \leq 1_{\{\mathcal{B}(0,1)\}}$. Define $F(t,x) = \int_{\mathbb{R}^6}\phi(v - x)\mu_t(dv,dr)$. Then by Remark 2.7 in \cite{sund}, $t \to F$ is continuous. Further, using the Lipschitz property of $\phi$, we have $\sup_t|F(t,x) - F(t,y)| \leq C|x - y|$ so that we have continuity in $x$. Since $F(t,x) \geq m_t(B(x,\frac{1}{2}))$ by definition, we deduce $F(t,x) > 0$ for all $t > 0$ and $x \in \mathbb{R}^3$ using the support of the marginal distribution $m_t$. Given $[t_0,t_1]\times \mathcal{S}(0,2)$ is compact and $F$ is continuous in both variables, we have $\inf_{t \in [t_0,t_1],x\in\mathcal{S}(0,2)}F(t,x) > 0$, and thus $\inf_{t \in [t_0,t_1],x\in\mathcal{S}(0,2)} m_t(\mathcal{B}(x,1)) \geq \inf_{t \in [t_0,t_1],x\in\mathcal{S}(0,2)} F(t,x) > 0$.
\item [Step 2.] Now we must check that for $u$ and $\xi \in \mathbb{R}^3$, we can find $v_{w,\xi} \in \mathcal{S}(0,2)$ such that $\mathcal{B}(v_{w,\xi},1) \subset K(u,\xi)$. Using Proposition 4.2 from \cite{four}, this is immediate, barring the condition on $\sigma$. From Fournier, we have $|w - v| \geq 1$ for $v \in \mathcal{B}(v_{w,\xi})$ with $v_{w,\xi} = -2\text{sg}(<w,\xi>)\xi/|\xi|$, which implies $w \neq v$ for any $v$ in the ball. Thus $\sigma(w - v) > 0$ for all $v \in B(v_{w,\xi},1)$. By the continuity of $\sigma$, we have some $m > 0$, such that $\sigma(v - w) \geq m$ for each $v$ in the ball.
\item [Step 3.] By containment from Step 2,
\begin{equation*}
    \inf_{t\in[t_0,t_1]u,\xi \in \mathbb{R}^3}m_t(K(w,\xi)) \geq \inf_{t\in[t_0,t_1],x\in\mathcal{S}(0,2)}m_t(\mathcal{B}(x,1)) > 0
\end{equation*}
where positiveness was given in Step 1.
\end{itemize}
\end{proof}
As the velocity process is driven by conditions on both velocity and position, this estimate must be further refined for the joint distribution. Namely, $\beta$ must keep its distance from zero as well. The established estimate extends to the following lemma. 
\begin{lemma}
There exists some $m > 0$ such that 
\begin{equation*}
    q_{t_0,t_1} : = \inf_{t\in[t_0,t_1]\\u,v_0,r_0}\mu_t(K(u,v_0)\times B^m_\beta(r_0)) > 0
\end{equation*}
where $B^m_\beta(r_0) = \{r \in \mathbb{R}^3:\beta(r - r_0) \geq m > 0\}$
\end{lemma}
\begin{proof}
    Suppose there were not such an $m$. Then we would have, for $t \in [t_0,t_1]$ and $u,v_0,r_0 \in \mathbb{R}^3$,
    \begin{equation*}
        \mu_t(K(u,v_0) \times B_\beta^m(r_0)) = 0.
    \end{equation*}
    for each positive $m$. However, if $\{m_n\}$ is some decreasing sequence of positive real numbers with $m_n \to 0$, then $B_{\beta}^{m_n}(r_0) \to \mathbb{R}^3$ as $n \to \infty$ since $\beta$ is a non-negative function. Thus, using continuity from below, 
    \begin{equation*}
        \lim_{n\to\infty}\mu_t(K(u,v_0) \times B_\beta^{m_n}(r_0)) = \mu_t(K(u,v_0) \times \mathbb{R}^3) = m_t(K(u,v_0))
    \end{equation*}
    but this $m_t(K(u,v_0))$ is positive by the previous lemma. This is a contradiction as we assumed $\mu_t(K(u,v_0) \times B_\beta^{m_n}(r_0)) = 0$ for each $n$. Thus, it must be that there is some positive $m$ such that $\mu_t(K(u,v) \times B_\beta^m(r_0)) > 0$ for all $t \in [t_0,t_1]$.   
\end{proof}
Further, due to the like distribution of $(R_t,V_t)$ and $(u_t,q_t)$, we can take the proposition and remark to give the following remark.
\begin{remark}
By the above,
\begin{equation*}
\inf_{t\in[t_0,t_1],u,\xi,r_0\in\mathbb{R}^3}\eta_t(K(u,\xi)\times\{\beta(r - r_0) \geq m > 0\}) = q_{t_0,t_1} > 0.
\end{equation*}
where $q_{t_0,t_1}$ is some positive constant possibly smaller than $c_{t_0,t_1}$.
\end{remark}
\section{Estimates for the Boltzmann-Enskog process}
\textbf{A Lower Bound}\\
In the previous results, it was sufficient to assume $\int_0^\pi \theta Q(d\theta) < \infty$, without giving an explicit form for $Q$. However, determining estimates related to the process proves somewhat difficult without giving a more concrete definition for $Q(d\theta)$. The estimate established in this section will show a clear need for this behavior of $Q$, as the convergence rates would not be ascertainable without such behavior.\\
Initially, we make some observations relevant to additive processes driven by Poisson random measures. Define $V_t^\varepsilon$ as introduced earlier. Then for $V_t^\varepsilon$ with $\varepsilon \in (0,t \wedge 1)$, there exists a filtration $(\mathcal{F}_t)$ and an adapted Poisson measure $M$ with the same intensity as the $N$ such that 
\begin{equation*}
    V_t^\varepsilon = V_{t-\varepsilon} + \int_{t-\varepsilon}^t\int_{\Xi\times\chi\times\mathbb{R}_+}\hat{\alpha}(V_{t-\varepsilon},R_{t-\varepsilon},u_s(\eta),q_s(\eta),\theta,\xi,z)dM(s,\theta,\xi,\eta,z).
\end{equation*}
When we say the same intensity, we of course mean
\begin{equation*}
    d\hat{M}(s,\theta,\xi,\eta,z) = dsQ(d\theta)d\xi d\eta_s(dq,du) dz.
\end{equation*}
We can further write $V_t^\varepsilon = U_t^\varepsilon + W_t^\varepsilon$ where
\begin{equation*}
    U_t^\varepsilon = \int_{t-\varepsilon}^t\int_{\Xi\times\chi\times\mathbb{R}_+}\hat{\alpha}(V_{t-\varepsilon},R_{t-\varepsilon},u_s(\eta),q_s(\eta),\theta,\xi,z)1_{\{\theta < \varepsilon^{1/\nu}\}}dM(s,\theta,\xi,\eta,z)
\end{equation*}
and 
\begin{equation*}
    W_t^\varepsilon = V_{t-\varepsilon} + \int_{t-\varepsilon}^t\int_{\Xi\times\chi\times\mathbb{R}_+}\hat{\alpha}(V_{t-\varepsilon},R_{t-\varepsilon},u_s(\eta),q_s(\eta),\theta,\xi,z)1_{\{\theta \geq \varepsilon^{1/\nu}\}}dM(s,\theta,\xi,\eta,z)
\end{equation*}
so that $U_t^\varepsilon$ and $W_t^\varepsilon$ are independent conditionally on $\mathcal{F}_{t-\varepsilon}$.\\
Lastly, for $\kappa \in \mathbb{R}^3$, $E[e^{i<\kappa,U_t^\varepsilon>}|\mathcal{F}_{t-\varepsilon}] = \text{exp}(-\psi_{\varepsilon,t,V_{t-\varepsilon},R_{t-\varepsilon}}(\kappa))$ where for $v_0,r_0 \in \mathbb{R}^3$,
\begin{equation}
\label{psi-char}
    \psi_{\varepsilon,t,v_0,r_0}(\kappa) = \int_{t-\varepsilon}^t\int_{\chi}\int_{0}^{\varepsilon^{1/\nu}}\int_{S^2}(1 - e^{i<\kappa,\alpha(v_0,u_s,\theta,\xi)>})\sigma(v_0 - u_s)\beta(r_0 - q_s)d\xi Q(d\theta)d\eta ds
\end{equation}
This is since the approximating process defined above is an additive process conditionally on $\mathcal{F}_{t-\varepsilon}$ as the filtration concerns the jumps, with the integrator $M$ allowed due to the jacobian 1 transformation $\xi_0$. This additive property can be seen when conditioning $V^\varepsilon_t - V_{t-\varepsilon} - V^\varepsilon_s + V_{s-\varepsilon}$ with respect to the filtration $(\mathcal{F}_t)_t$ generated by the time differences of the Poisson random measure $M$ up to time $t - \varepsilon$. As the integral concerns only time differences in $(s - \varepsilon,s)$ and $(t-\varepsilon,t)$, we conclude it has independent increments. Stochastic continuity of paths is due to the expectation bounds on time differences established in calculations made in the proof of Proposition \ref{approx-def}. In fact, 
\begin{equation}
\label{stoch-cty}
E|V_t^\varepsilon - V_{t-\varepsilon} - V_s^\varepsilon - V_{s-\varepsilon}| \leq O(t - s)
\end{equation}
 This is rigorously shown, with much of the effort being eased by the expectation estimates established in Proposition \ref{approx-def}. 
\begin{lemma}
The approximating process $V_t^\varepsilon$ and the frozen time integral,
\begin{equation*}
    \int_{t-\varepsilon}^t\int_{\Xi\times\chi\times\mathbb{R}_+}\hat{\alpha}(V_{t-\varepsilon},R_{t-\varepsilon},u_s(\eta),q_s(\eta),\theta,\xi,z)dM(s,\theta,\xi,\eta,z)
\end{equation*}
are stochastically continuous processes.
\end{lemma}
\begin{proof}
Firstly, let $|t - s| < \varepsilon$ so that $s > t-\varepsilon$. Consider
\begin{equation}
\label{C-t-s}
\begin{split}
    C(t,s) := &E\left|\int_{t-\varepsilon}^t\int_{\Xi\times\chi\times\mathbb{R}_+}\hat{\alpha}(V_{t-\varepsilon},R_{t-\varepsilon},u_r,q_r,\theta,\xi,z)dM(r,\theta,\xi,\eta,z)\right.\\
    &\left.-\int_{s-\varepsilon}^s\int_{\Xi\times\chi\times\mathbb{R}_+}\hat{\alpha}(V_{s-\varepsilon},R_{s-\varepsilon},u_r,q_r,\theta,\xi,z)dM(r,\theta,\xi,\eta,z)\right| 
\end{split}
\end{equation}
By the triangle inequality, one has 
\begin{equation}
\label{C-t-s 123 bound}
    C(t,s) \leq I_1 + I_2 + I_3
\end{equation}
The $I_k$ are defined as 
\begin{equation}
\label{I1}
    I_1 := E\left|\int_{s}^t\int_{\Xi\times\chi\times\mathbb{R}_+}\hat{\alpha}(V_{t-\varepsilon},R_{t-\varepsilon},u_r,q_r,\theta,\xi,z) dM(r,\theta,\xi,\eta,z)\right|
\end{equation}
\begin{align}
\label{I2}
    I_2 := E\left|\int_{t-\varepsilon}^s\int_{\Xi\times\chi\times\mathbb{R}_+}\hat{\alpha}(V_{t-\varepsilon},R_{t-\varepsilon},u_r,q_r,\theta,\xi,z) - \hat{\alpha}(V_{s-\varepsilon},R_{s-\varepsilon},u_r,q_r,\theta,\xi,z) dM(r,\theta,\xi,\eta,z)\right|
\end{align}
\begin{equation}
    \label{I3}
   I_3 := E\left|\int_{s-\varepsilon}^{t - \varepsilon}\int_{\Xi\times\chi\times\mathbb{R}_+}\hat{\alpha}(V_{s-\varepsilon},R_{s-\varepsilon},u_r,q_r,\theta,\xi,z) dM(r,\theta,\xi,\eta,z)\right|
\end{equation}
For (\ref{I1}) and (\ref{I3}), the work to establish (\ref{lipschitz-in-time}) from Proposition \ref{approx-def} can be redone to establish 
\begin{equation}
    \label{I1-I3-bound}
    I_1 + I_3 \leq K_t(t - s)
\end{equation}
using the boundedness of $\beta$, and the finite values of $E|V_{t-\varepsilon}|^{1 + \gamma}$ and $E|V_{s - \varepsilon}|^{1 + \gamma}$. The more difficult task is to obtain similar bounds for $I_2$ as defined in (\ref{I2}). Write
\begin{equation}
    \label{I2 j bound}
    I_2 \leq J_\beta + J_\sigma + J_\alpha
\end{equation}
where the $J$ terms, given below, denote integral terms dealing with differences on $\beta$, $\sigma$ and $\alpha$ respectively. Define
\begin{equation}
    \label{Jbeta}
    J_\beta := \int_{t-\varepsilon}^s\int_{\Xi\times\chi}E\left[|\beta(R_{t-\varepsilon} - q_r) - \beta(R_{s-\varepsilon} - q_r)||\alpha(V_{t-\varepsilon},u_r,\theta,\xi)\sigma(V_{t - \varepsilon} - u_s)|\right]d\eta Q(d\theta) d\xi dr
\end{equation}
Applying the Cauchy-Schwarz inequality, Lipschitz property of $\beta$, and (\ref{R-lip}) in (\ref{Jbeta}) gives
\begin{equation}
    \label{final jbeta bound}
    J_\beta \leq O(t - s).
\end{equation}
Notice that the $\varepsilon$ vanishes upon utilizing (\ref{R-lip}). Next, is to examine $J_\sigma$. Recall that $\beta$ is bounded, so that we can instead bound 
\begin{equation}
    \label{Jsigma}
    J_\sigma := \int_{t-\varepsilon}^s\int_{\Xi\times\chi}E\left[|\alpha(V_{t-\varepsilon},u_r,\theta,\xi)||\sigma(V_{t - \varepsilon} - u_s) - \sigma(V_{s-\varepsilon})|\right]d\eta Q(d\theta) d\xi dr
\end{equation}
As in steps 3 and 4 of the proof of Proposition \ref{approx-def}, we separate by the $\gamma \in (0,1]$ case and the $\gamma \in (1,2)$ case. For $\gamma \in (0,1]$, it is sufficient to utilize H\"{o}lder continuity and obtain
\begin{equation}
    \label{jsigma 1}
    J_\sigma \leq \int_{t-\varepsilon}^s\int_{\Xi\times\chi}E\left[|\alpha(V_{t-\varepsilon},u_r,\theta,\xi)||V_{t - \varepsilon} - V_{s-\varepsilon}|^\gamma\right]d\eta Q(d\theta) d\xi dr
\end{equation}
By the H\"{o}lder inequality with $p = 1 + \gamma$ and $q = \frac{1 + \gamma}{\gamma}$ and (\ref{vt-vu 1 + gamma}), it follows that $J_\sigma \leq O(t-s)$. For the $\gamma \in (1,2)$, we must instead use 
\begin{equation}
    \label{jsigma 2}
    J_\sigma \leq \int_{t-\varepsilon}^s\int_{\Xi\times\chi}E\left[|\alpha(V_{t-\varepsilon},u_r,\theta,\xi)||V_{t - \varepsilon} - V_{s-\varepsilon}|(|V_{t-\varepsilon} - u_r|^{\gamma - 1} + |V_{s-\varepsilon} - u_r|^{\gamma - 1})\right]d\eta Q(d\theta) d\xi dr
\end{equation}
Applying (\ref{alpha-lip}) to bound $\alpha$, we need to estimate
\begin{equation*}
    E[|V_{t-\varepsilon} - V_{s-\varepsilon}|(|V_{t - \varepsilon} - u_r|^\gamma + |V_{t-\varepsilon} - u_r||V_{s-\varepsilon} - u_r|^{\gamma - 1}]
\end{equation*}
which can be separated as
\begin{equation*}
    E[|V_{t-\varepsilon} - V_{s-\varepsilon}||V_{t - \varepsilon} - u_r|^\gamma] + E[|V_{t-\varepsilon} - V_{s-\varepsilon}||V_{t-\varepsilon} - u_r||V_{s-\varepsilon} - u_r|^{\gamma - 1}]
\end{equation*}
By the H\"{o}lder inequality with $p = 1 + \gamma$ and $q = \frac{1 + \gamma}{\gamma}$ and (\ref{vt-vu 1 + gamma}), it follows that $J_\sigma \leq O(t-s)$. In total, whether using the initial bound of (\ref{jsigma 1}) for $\gamma \in (0,1]$ or the initial bound of (\ref{jsigma 2}) for $\gamma \in (1,2)$, 
\begin{equation}
\label{jsigma final bound}
    J_\sigma \leq O(t-s)
\end{equation}
Finally examine $J_\alpha$, defined as 
\begin{equation}
    \label{jalpha-initial}
    J_\alpha := E\left[\left|\int_{t-\varepsilon}^s\int_{\Xi\times\chi}E\left[(\alpha(V_{t-\varepsilon},u_r,\theta,\xi) - \alpha(V_{s-\varepsilon},u_r,\theta,\xi))\sigma(V_{s-\varepsilon} - u_r)|\right]d\eta Q(d\theta) d\xi dr\right|\right]
\end{equation}
We use the jacobian 1 transform in Lemma \ref{xi-lemma} to the first $\alpha$ function in (\ref{jalpha-initial}). Bringing in the absolute value to the integrand gives
\begin{equation}
    \label{jalpha}
    J_\alpha \leq \int_{t-\varepsilon}^s\int_{\Xi\times\chi}E\left[|\alpha(V_{t-\varepsilon},u_r,\theta,\xi + \xi_0) - \alpha(V_{s-\varepsilon},u_r,\theta,\xi)|\sigma(V_{s-\varepsilon} - u_r)|\right]d\eta Q(d\theta) d\xi dr
\end{equation}
 Applying (\ref{alph-cont}) and (\ref{sigmaBound}) shows
 \begin{equation}
     \label{jalpha1}
     J_\alpha \leq E\int_{t-\varepsilon}^s\int_{\Xi\times\chi}\frac{1}{2}\theta E|V_{t-\varepsilon} - V_{s-\varepsilon}| + E[|V_{t-\varepsilon} - V_{s-\varepsilon}||V_{s-\varepsilon} - u_r|^\gamma]d\eta Q(d\theta) d\xi dr
 \end{equation}
 From (\ref{vt - vu ineq}), we have $E|V_{t-\varepsilon} - V_{s-\varepsilon}| \leq K(t-s)$ for some constant. Using the H\"{o}lder inequality on the second expectation with $p = 1 + \gamma$ and $q = \frac{1 + \gamma}{\gamma}$ once more gives it to be a term bounded by some constant multiple of $(t - s)$ by (\ref{vt-vu 1 + gamma}). Thus, it can be stated that 
 \begin{equation}
     \label{jalpha final bound}
     J_\alpha \leq O(t-s)
 \end{equation}
 Gathering the bounds (\ref{final jbeta bound}) - (\ref{jalpha final bound}) on the $J$ terms (\ref{Jbeta}) - (\ref{jalpha}), and recalling (\ref{I2 j bound}) gives 
 \begin{equation}
     \label{I2 order}
     I_2 \leq O(t-s)
 \end{equation}
 for some term $O(t-s)$ which goes to zero as $t \to s$. Collecting (\ref{I1-I3-bound}) and (\ref{I2 order}), and recalling (\ref{C-t-s 123 bound}) gives
 \begin{equation}
     \label{c-t-s order}
     C(t,s) \leq O(t-s)
 \end{equation}
 The definition of $C(t,s)$ given in (\ref{C-t-s}) shows (\ref{stoch-cty}) as desired. Using the Markov inequality along with (\ref{c-t-s order}) gives the stochastic continuity of the frozen time integral in the approximating process. Namely, for any $\varepsilon' > 0$, 
 \begin{equation*}
 P(|V_t^\varepsilon - V_{t-\varepsilon} - V_s^\varepsilon + V_{s-\varepsilon}| > \varepsilon') \leq \frac{C(t,s)}{\varepsilon'} \leq \frac{O(t-s)}{\varepsilon'} \to 0
 \end{equation*}
 as $t \to s$. This in turn gives stochastic continuity of the approximating process itself as well, as if one defines $I_t := V_t^\varepsilon - V_{t-\varepsilon}$, the inequality $E|V_t^\varepsilon - V_s^\varepsilon| \leq E|V_{t-\varepsilon} - V_{s-\varepsilon}| + E|I_t - I_s|$ implies $E|V_t^\varepsilon - V_s^\varepsilon| \leq O(t-s)$ which in turn gives stochastic continuity by the Markov inequality.
 \end{proof}
With the verification that the integral involved in $V_t^\varepsilon$ is additive, its law can be studied using the L\'{e}vy-Khintchine formula for additive processes (see Chapter 2, Section 2 of \cite{jacod}). The following bound is established in relation to the approximating process's transform.
\begin{lemma}
\label{Lower-BD}
Define $\psi_{\varepsilon,t,v_0,r_0}$ as in (\ref{psi-char}). Then, for all $\kappa \in \mathbb{R}^3$ and all $0 < t_0 \leq t - \varepsilon \leq t \leq t_1$, with $\varepsilon \in (0,1)$,
\begin{equation*}
    \text{Re}\,\psi_{\varepsilon,t,v_0,r_0}(\varepsilon^{-1/\nu}\kappa) \geq Cq_{t_0,t_1}\left[|\kappa|^2 \wedge |\kappa|^\nu\right]
\end{equation*}
for $\gamma \in (0,2)$.
\end{lemma}
\begin{proof}
Note by definition,
\begin{equation*}
\begin{split}
    \text{Re}\,\psi_{\varepsilon,t,v_0,r_0}(\varepsilon^{-1/\nu}\kappa) & = \int_{t-\varepsilon}^t\int_{\chi}\int_{0}^{\varepsilon^{1/\nu}}\int_{S^2}(1 - \cos(\varepsilon^{-1/\nu}<\kappa,\alpha(v_0,u_s,\theta,\xi)>)\\&\sigma(v_0 - u_s)\beta(r_0 - q_s) d\xi b(\theta)d\theta d\eta ds
\end{split}
\end{equation*}
Now we focus on the $S^2$ integral.
\begin{equation*}
    \int_{S^2}1 - \cos(\varepsilon^{-1/\nu}<\kappa,\alpha(v_0,u_s,\theta,\xi)>)d\xi
\end{equation*}
Using \eqref{alpha-defn} and a trigonometric identity allows us to write
\begin{equation*}
    \alpha(v_0,u_s,\theta,\xi) = \frac{\cos(\theta) - 1}{2}(v_0 - u_s) + \frac{\sin(\theta)}{2}\Gamma(v_0 - u_s,\xi).
\end{equation*}
By a sum-to-product identity, we have 
\begin{equation*}
    \begin{split}
    & \int_{S^2}1 - \cos(\varepsilon^{-1/\nu}\left<\kappa,\alpha(v_0,u_s,\theta,\xi)\right>)d\xi\\
    & = \int_{S^2}1 - \cos(\varepsilon^{-1/\nu}(\cos(\theta) - 1)\left<\kappa,v_0 - u_s\right>/2)\cos(\varepsilon^{-1/\nu}\sin(\theta)\left<\kappa,\Gamma(v_0 - u_s,\xi)\right>/2)\\
    & + \sin(\varepsilon^{-1/\nu}(\cos(\theta) - 1)\left<\kappa,v_0 - u_s\right>/2)\sin(\varepsilon^{-1/\nu}\sin(\theta)\left<\kappa,\Gamma(v_0 - u_s,\xi)\right>/2)d\xi
\end{split}
\end{equation*}
As $\xi$ is in $S^2$, there exists an angle $\phi \in [0,2\pi)$ for which $\xi = (\cos(\phi),\sin(\phi))$. This means we can consider the above as a one-dimensional integration over an angle $\phi$ rather than over a multi-dimensional space. With respect to $\phi$, the argument of the second sine function is $2\pi$ periodic, thus the above integral is equivalent to
\begin{equation}
\label{1-cos}
    \int_{S^2}1 - \cos(\varepsilon^{-1/\nu}(\cos(\theta) - 1)\left<\kappa,v_0 - u_s\right>/2)\cos(\varepsilon^{-1/\nu}\sin(\theta)\left<\kappa,\Gamma(v_0 - u_s,\xi)\right>/2)d\xi
\end{equation}
Now we take a lower bound. Since $|\cos(x)| \leq 1$, the above is always positive, and subtracting $|\cos(x)|$ instead of $\cos(x)$ will maintain positivity. Thus (\ref{1-cos}) is greater than or equal to
\begin{equation}
    \label{1-|cos|}
        \int_{S^2}1 - |\cos(\varepsilon^{-1/\nu}(\cos(\theta) - 1)\left<\kappa,v_0 - u_s\right>/2)\cos(\varepsilon^{-1/\nu}\sin(\theta)\left<\kappa,\Gamma(v_0 - u_s,\xi)\right>/2)|d\xi
\end{equation}
and this lower bound is still positive. Further, this bound is greater than
\begin{equation*}
        \int_{S^2}1 - |\cos(\varepsilon^{-1/\nu}\sin(\theta)\left<\kappa,\Gamma(v_0 - u_s,\xi)\right>/2)|d\xi
\end{equation*}
as we maximize one of the cosine functions. Now, noting that $1 - |\cos(x)| \geq \frac{x^2}{4}$ and $|\sin(x)| \geq \frac{|x|}{2}$ for $x \in [-1,1]$ and since $|\sin(x)| \leq |x|$ for all $x \in \mathbb{R}$, we have
\begin{equation*}
\begin{split}
        & \int_{S^2}1 - |\cos(\varepsilon^{-1/\nu}\sin(\theta)\left<\kappa,\Gamma(v_0 - u_s,\xi)\right>/2)|d\xi\\
        & \geq \int_{S^2} \frac{\varepsilon^{-2/\nu}\sin^2(\theta)\inprod{\kappa}{\Gamma(v_0 - u_s,\xi)}^2}{16}1_{\{|\inprod{\kappa}{\Gamma(v_0-u_s,\xi)}| \sin(\theta)\leq 2\varepsilon^{1/\nu}\}}d\xi\\
        & \geq \int_{S^2} \frac{\varepsilon^{-2/\nu}\theta^2\inprod{\kappa}{\Gamma(v_0 - u_s,\xi)}^2}{64}1_{\{|\theta| \leq 2\varepsilon^{1/\nu}/|\inprod{\kappa}{\Gamma(v_0 - u_s,\xi)}|\}}d\xi\\
\end{split}
\end{equation*}
Thus, we have the following lower bound.
\begin{equation*}
     \begin{split}
     &\text{Re}\,\psi_{\varepsilon,t,v_0,r_0}(\varepsilon^{-1/\nu}\kappa) \geq \int_{t-\varepsilon}^t\int_{\chi}\int_{0}^{\varepsilon^{1/\nu}}\int_{S^2}\frac{\varepsilon^{-2/\nu}\theta^2\inprod{\kappa}{\Gamma(v_0 - u_s,\xi)}^2}{64}\\
     & \times 1_{\{|\theta| \leq 2\varepsilon^{1/\nu}/|\inprod{\kappa}{\Gamma(v_0 - u_s,\xi)}^2|\}}\sigma(v_0 - u_s)\beta(r_0 - q_s) d\xi b(\theta)d\theta d\eta ds\\
     & = \frac{c}{\varepsilon^{2/\nu}} \int_{t-\varepsilon}^t\int_{\chi}\int_{0}^{\varepsilon^{1/\nu}}\int_{S^2}\theta^2\inprod{\kappa}{\Gamma(v_0 - u_s,\xi)}^2\\
     & \times 1_{\{|\theta| \leq 2\varepsilon^{1/\nu}/|\inprod{\kappa}{\Gamma(v_0 - u_s,\xi)}|\}}\sigma(v_0 - u_s)\beta(r_0 - q_s) d\xi b(\theta)d\theta d\eta
    \end{split}
\end{equation*}
We further bound below by examining the $Q(d\theta)$ integral. Namely, 
\begin{equation*}
    \int_0^{\varepsilon^{1/\nu}}\theta^21_{\{|\theta| \leq 2\varepsilon^{1/\nu}/|\inprod{\kappa}{\Gamma(v_0 - u_s,\xi)}|\}}b(\theta)d\theta .
\end{equation*}
is larger than 
\begin{equation*}
    c_0\int_0^{\varepsilon^{1/\nu}}\theta^{1-\nu} 1_{\{|\theta| \leq 2\varepsilon^{1/\nu}/|\inprod{\kappa}{\Gamma(v_0 - u_s,\xi)}|\}} d \theta
\end{equation*}
which is equivalent to
\begin{equation*}
    c_0\left[\varepsilon^{1/\nu} \wedge \frac{2\varepsilon^{1/\nu}}{|\inprod{\kappa}{\Gamma(v_0 - u_s,\xi)}|}\right]^{2-\nu}
\end{equation*}
So, we have
\begin{equation*}
    \begin{split}
        & \frac{c}{\varepsilon^{2/\nu}} \int_{t-\varepsilon}^t\int_{\chi}\int_{0}^{\varepsilon^{1/\nu}}\int_{S^2}\theta^2\inprod{\kappa}{\Gamma(v_0 - u_s,\xi)}^2\\
     & \times 1_{\{|\theta| \leq 2\varepsilon^{1/\nu}/|\inprod{\kappa}{\Gamma(v_0 - u_s,\xi)}|\}}\sigma(v_0 - u_s)\beta(r_0 - q_s) d\xi b(\theta)d\theta d\eta_s ds\\
     & \geq \frac{c}{\varepsilon^{2/\nu}} \int_{t-\varepsilon}^t\int_{\chi}\int_{S^2}\left[\varepsilon^{1/\nu} \wedge \frac{2\varepsilon^{1/\nu}}{|\inprod{\kappa}{\Gamma(v_0 - u_s,\xi)}|}\right]^{2-\nu}\inprod{\kappa}{\Gamma(v_0 - u_s,\xi)}^2\\
     & \times \sigma(v_0 - u_s)\beta(r_0 - q_s) d\xi d\eta_s ds\\
     & = \frac{c}{\varepsilon} \int_{t-\varepsilon}^t\int_{\chi}\int_{S^2}\left[|\inprod{\kappa}{\Gamma(v_0 - u_s,\xi)}|^2 \wedge |\inprod{\kappa}{\Gamma(v_0 - u_s,\xi)}|^\nu\right]\\
     & \times \sigma(v_0 - u_s)\beta(r_0 - q_s) d\xi d\eta_s ds
    \end{split}
\end{equation*}
We once more note that our $S^2$ integral can be considered as an integral over $\phi \in [0,2\pi)$ so that we can apply the parameterization of Fournier \cite{four}. We briefly state this parameterization. For $X \in \mathbb{R}^3 \setminus \{0\}$, we are able to define $I(X)$ and $J(X)$ so that $\left(\frac{X}{|X|},\frac{I(X)}{|X|},\frac{J(X)}{|X|}\right)$ is an orthonormal basis in such a way that $X \to (I(X),J(X))$ is measurable. Set $I(0) = J(0) = 0$. For $X,v,v_* \in \mathbb{R}^3$, $\theta \in [0,\pi)$, and $\phi \in [0,2\pi)$, we set 
\begin{equation}
\label{four-paramaterization}
    \begin{cases}
        &\Gamma(X,\phi) := (\cos \phi) I(X) + (\sin \phi)J(X),\\
        &v'(v,v_*,\theta,\phi) := v - \frac{1 - \cos\theta}{2}(v - v_*) + \frac{\sin\theta}{2}\\
        &\alpha(v,v_*,\theta,\phi) := v'(v,v_*,\theta,\phi) - v
    \end{cases}
\end{equation}
The utility of the parameterization (\ref{four-paramaterization}) is that we can use some symmetry arguments in integration. The following lemma is Remark 3.1 in \cite{four}, where it is proven succinctly.
\begin{lemma}
    For any measurable non-negative function $F: \mathbb{R} \to \mathbb{R}$, and $X,Y \in \mathbb{R}^3$,
    \begin{equation}
    \label{fournier-symmetry}
        \int_{0}^{2\pi}F(\left<Y,\Gamma(X,\phi)\right>)d\phi = \int_{0}^{2\pi}F(\left<X,\Gamma(Y,\phi)\right>)d\phi
    \end{equation}
\end{lemma}
Given that $\xi \in S^2$, we can parameterize it using $\phi \in [0,2\pi)$. Thus, we may instead write (\ref{fournier-symmetry}) as
\begin{equation}
\label{symmetry-equation}
    \int_{S^2}F(\left<Y,\Gamma(X,\xi)\right>)d\xi = \int_{S^2}F(\left<X,\Gamma(Y,\xi)\right>)d\xi
\end{equation}
Utilizing (\ref{symmetry-equation}) gives the following equality,
\begin{equation*}
    \begin{split}
        & \frac{c}{\varepsilon} \int_{t-\varepsilon}^t\int_{\chi}\int_{S^2}\left[|\inprod{\kappa}{\Gamma(v_0 - u_s,\xi)}|^2 \wedge |\inprod{\kappa}{\Gamma(v_0 - u_s,\xi)}|^\nu\right]\sigma(v_0 - u_s)\beta(r_0 - q_s) d\xi d\eta_s ds\\
     & = \frac{c}{\varepsilon} \int_{t-\varepsilon}^t\int_{\chi}\int_{S^2}\left[|\inprod{v_0 - u_s}{\Gamma(\kappa,\xi)}|^2 \wedge |\inprod{v_0 - u_s}{\Gamma(\kappa,\xi)}|^\nu\right]\sigma(v_0 - u_s)\beta(r_0 - q_s) d\xi d\eta_s ds\\
    \end{split}
\end{equation*}
Now, we use the constructed set from Proposition \ref{geometric prop}. First, we note that $|\Gamma(\kappa,\xi)| = |\kappa|$ by construction. Now, we allow $u_s \in K(v_0,\Gamma(\kappa,\xi))$ and $q_s \in \{\beta(r_0 - q_s) \geq m > 0\}$. Then $|\inprod{v_0 - u_s}{\Gamma(\kappa,\xi)}| \geq |\Gamma(\kappa,\xi)| = |\kappa|$.
\begin{equation*}
    \begin{split}
             & = \frac{c}{\varepsilon} \int_{t-\varepsilon}^t\int_{\chi}\int_{S^2}\left[|\inprod{v_0 - u_s}{\Gamma(\kappa,\xi)}|^2 \wedge |\inprod{v_0 - u_s}{\Gamma(\kappa,\xi)}|^\nu\right]\\
     & \times \sigma(v_0 - u_s)\beta(r_0 - q_s) d\xi d\eta_s ds\\
     & \geq \frac{cm^2}{\varepsilon} \int_{t-\varepsilon}^t\int_{S^2}\left[|\kappa|^2 \wedge |\kappa|^\nu\right] d\xi d\eta_s(K(v_0,\Gamma(\kappa,\xi))\times\{\beta(r_0 - q_s) \geq m > 0\}) ds\\
     & \geq Cq_{t_0,t_1}\left[|\kappa|^2 \wedge |\kappa|^\nu\right]\\
    \end{split}
\end{equation*}
\end{proof}
In subsequent lemmas, a dependence on $\varepsilon$ would be a hindrance in establishing upper bounds related to the approximating process. However, the subsequent results do depend on $\varepsilon$, which is helped by the current bound's independence. The end goal is to assimilate the bounds to give the final estimate needed to establish the existence of a density function. 
\\ \textbf{Fourier Transform Results}\\
To gain information on the Fourier transform of the approximating process, we first recall Lemma 7.2 from \cite{four}, where we refer the reader to for the proof.
%le7.2 #&#
\begin{lemma}
\label{ssw1}
Let $\lambda$ be a nonnegative measure on $\mathbb{R}^3$ such that $\int_{\mathbb{R}^3}
|y| \lambda(dy)<\infty$
and consider
the infinitely divisible distribution $k$ with Fourier transform
\begin{eqnarray*}
\hat k(\xi):=\int_{\mathbb{R}^3}e^{i \langle\xi,x \rangle
}k(dx)=\exp\bigl(-
\Phi(\xi)\bigr) \qquad\mbox{with } \Phi(\xi)=\int_{\mathbb{R}^3}
\bigl(1-e^{i \langle\xi
,y \rangle} \bigr) \lambda(dy).
\end{eqnarray*}
If the right-hand side of the following inequality is finite, then $k$ has a
density (still denoted
by $k$) and
\[
\norm{\nabla k}_{L^1(\mathbb{R}^3)} \leq C \bigl(1+m_1^4(
\lambda)+m_4(\lambda) \bigr) \int_{\mathbb{R}^3}
e^{- \text{Re}\Phi(\xi)}(1+ |\xi|)\,d\xi, %
\]
where $m_n(\lambda)=\int_{\mathbb{R}^3} |y|^n \lambda(dy)$ and $C$ is a
universal constant.
\end{lemma}

\begin{lemma}
\label{Upper-BD}
Recall
\begin{equation*}
        \psi_{\varepsilon,t,v_0,r_0}(\kappa) = \int_{t-\varepsilon}^t\int_{\chi}\int_{0}^{\varepsilon^{1/\nu}}\int_{S^2}(1 - e^{i<\kappa,\alpha(v_0,u_s,\theta,\xi)>})\sigma(v_0 - u_s)\beta(r_0 - q_s)d\xi Q(d\theta)d\eta ds
\end{equation*}
as defined in the previous section. Consider, $g_{\varepsilon,t,v_0,r_0} \in \mathcal{P}(\mathbb{R}^3)$ such that $\widehat{g_{\varepsilon,t,v_0,r_0}}(\kappa) = \text{exp}(-\psi_{\varepsilon,t,v_0,r_0}(\kappa))$. If $0 < t_0 \leq t - \varepsilon < t \leq t_1$ and $\varepsilon \in (0,1)$, then $g_{\varepsilon,t,v_0,r_0}$ has a density with the following estimate:
\begin{equation*}
    \norm{\nabla g_{\varepsilon,t,v_0,r_0}}_{L^1(\mathbb{R}^3)} \leq C\varepsilon^{-1/\nu}(1 + |v_0|^{\gamma + 4})
\end{equation*}
for $\gamma \in (0,2)$.
\end{lemma}
\begin{proof}
For $X_{\varepsilon,t,v_0,r_0}$, a $g_{\varepsilon,t,v_0,r_0}$ distributed variable, define
\begin{equation*}
    Y_{\varepsilon,t,v_0,r_0} = \varepsilon^{-\frac{1}{\nu}}X_{\varepsilon,t,v_0,r_0}.
\end{equation*}
Then the law $k_{\varepsilon,t,v_0,r_0}$ of $Y_{\varepsilon,t,v_0,r_0}$ satisfies $\widehat{k_{\varepsilon,t,v_0,r_0}}(\kappa) = \text{exp}(-\psi_{\varepsilon,t,v_0,r_0}(\varepsilon^{-1/\nu}\kappa)$ and $k_{\varepsilon,t,v_0,r_0}(x) = \varepsilon^{3/\nu}g_{\varepsilon,t,v_0,r_0}(\varepsilon^{1/\nu}x)$. Note
\begin{equation*}
    \norm{\nabla g_{\varepsilon,t,v_0,r_0}}_{L^1(\mathbb{R}^3)} \leq \varepsilon^{-1/\nu}\norm{\nabla k_{\varepsilon,t,v_0,r_0}}_{L^1(\mathbb{R}^3)} 
\end{equation*}
Define $\Phi_{\varepsilon,t,v_0,r_0}(\kappa) = \psi_{\varepsilon,t,v_0,r_0}(\varepsilon^{-1/\nu}\kappa)$. Then
\begin{equation*}
    \begin{split}
        & \Phi_{\varepsilon,t,v_0,r_0}(\kappa)\\ & = \int_{t-\varepsilon}^t\int_{\chi}\int_{0}^{\varepsilon^{1/\nu}}\int_{S^2}(1 - e^{i<\kappa,\varepsilon^{-1/\nu}\alpha(v_0,u_s,\theta,\xi)>})\sigma(v_0 - u_s)\beta(r_0 - q_s)d\xi b(\theta)d\theta d\eta ds\\
        & = \int_{\mathbb{R}^3}(1 - e^{i<\kappa,z>})\lambda_{t,\varepsilon,v_0,r_0}(dz)
    \end{split}
\end{equation*}
where the measure $\lambda_{t,\varepsilon,v_0,r_0}$ is defined by
\begin{equation*}
    \begin{split}
    & \int_{\mathbb{R}^3}F(z)\lambda_{t,\varepsilon,v_0,r_0}(dz)\\
    & = \int_{t-\varepsilon}^t\int_{\chi}\int_{0}^{\varepsilon^{1/\nu}}\int_{S^2}F\left(\frac{\alpha(v_0,u_s,\theta,\xi)}{\varepsilon^{1/\nu}}\right)\sigma(v_0 - u_s)\beta(r_0 - q_s)d\xi b(\theta)d\theta d\eta ds
    \end{split}
\end{equation*}
for all nonnegative measurable $F:\mathbb{R}^3 \to \mathbb{R}$. Lemma \ref{ssw1} implies
\begin{equation*}
    \begin{split}
        & \norm{k_{\varepsilon,t,v_0,r_0}}_{L^1(\mathbb{R}^3)}\\
        & \leq C(1 + m_1^4(\lambda_{t,\varepsilon,v_0,r_0}) + m_4(\lambda_{t,\varepsilon,v_0,r_0}))\int_{\mathbb{R}^3}e^{-\text{Re}\,\Phi_{\varepsilon,t,v_0(\kappa)}}(1 + |\kappa|)\,d\kappa\\
        & \leq C(1 + m_1^4(\lambda_{t,\varepsilon,v_0,r_0}) + m_4(\lambda_{t,\varepsilon,v_0,r_0}))\left(1 + \int_{|\kappa| \geq 1}e^{-\text{Re}\,\psi_{\varepsilon,t,v_0(\varepsilon^{-1/\nu}\kappa)}}|\kappa|d\kappa\right)
    \end{split}
\end{equation*}
provided the right side is finite. What we now do, is take one more H\"{o}lder inequality to instead use the following (larger) upper bound, and verify it is finite. 
\begin{equation*}
    \norm{k_{\varepsilon,t,v_0,r_0}}_{L^1(\mathbb{R}^3)} \leq C(1 + m_4(\lambda_{t,\varepsilon,v_0,r_0}))\left(1 + \int_{|\kappa| \geq 1}e^{-\text{Re}\,\psi_{\varepsilon,t,v_0(\varepsilon^{-1/\nu}\kappa)}}|\kappa|d\kappa\right)
\end{equation*}
Now, we work on a moment estimate.
\begin{equation}
    \begin{split}
        & m_4(\lambda_{t,\varepsilon,v_0,r_0}) \leq \int_{t-\varepsilon}^t\int_{\chi}\int_{0}^{\varepsilon^{1/\nu}}\int_{S^2}\frac{|\theta|^4|u_s - v_0|^4}{2^4\varepsilon^{4/\nu}}\sigma(v_0 - u_s)\beta(r_0 - q_s)d\xi b(\theta)d\theta d\eta ds\\
        & \leq C\int_{t - \varepsilon}^t\int_{\chi}\int_{0}^{\varepsilon^{1/\nu}}\frac{|\theta|^4(|v_0 - u_s|^4 + |v_0 - u_s|^{4 + \gamma})}{2^4\varepsilon^{4/\nu}}b(\theta)d\theta d\eta ds\\
        & \leq C\int_{t - \varepsilon}^t\int_{\chi}\int_{0}^{\varepsilon^{1/\nu}}\frac{|\theta|^{4 - 1 - \nu}(1 + |v_0 - u_s|^{4 + \gamma})}{\varepsilon^{4/\nu}}d\theta d\eta ds\\
        & = C\int_{t - \varepsilon}^t\int_{\chi}\frac{\varepsilon^{(4 + \nu)/\nu}(1 + |v_0 - u_s|^{4 + \gamma})}{\varepsilon^{4/\nu}}d\eta ds\\
        & \leq \frac{C}{\varepsilon}\int_{t - \varepsilon}^t\int_{\chi}1 + |v_0|^{4 + \gamma} + |u_s|^{4 + \gamma}d\eta ds\\
        & \leq C\sup_{s \in [t-\varepsilon,t]}\int_{\chi}(1 + |u_s|^{4 + \gamma} + |v_0|^{4 + \gamma})d\eta\\
    \end{split}
\end{equation}
Then note that since $4 + \gamma \leq \frac{8}{2-\gamma}$ for $\gamma \in (0,2)$:
\begin{equation*}
    C\sup_{s \in [t-\varepsilon,t]}\int_{\chi}(|u_s|^{4 + \gamma} + |v_0|^{4 + \gamma})d\eta \leq C_t(1 + |v_0|^{\gamma + 4})
\end{equation*}
where $C_{t}$ denotes a constant dependent on $t$. Now, from Lemma \ref{Lower-BD}, we have
\begin{equation*}
        \int_{|\kappa| \geq 1}e^{-\text{Re}\,\psi_{\varepsilon,t,v_0(\varepsilon^{-1/\nu}\kappa)}}|\kappa|d\kappa \leq C_{t_0,t_1}
\end{equation*}
where $C_{t_0,t_1}$ depends on $q_{t_0,t_1}$ from Lemma \ref{Lower-BD}. Thus,
\begin{equation*}
    \norm{\nabla k_{\varepsilon,t,v_0,r_0}}_{L^1(\mathbb{R}^3)} \leq C_{t_0,t_1}(1 + |v_0|^{\gamma + 4})
\end{equation*}
meaning 
\begin{equation*}
    \norm{\nabla g_{\varepsilon,t,v_0,r_0}}_{L^1(\mathbb{R}^3)} \leq \varepsilon^{-1/\nu}C_{t_0,t_1}(1 + |v_0|^{\gamma + 4})
\end{equation*}
\end{proof}
We note that the transform estimate does not directly depend on position, due to the bound $q_{t_0,t_1}$ being solely dependent on a lower bound for $\beta$. 
\section{Existence of Density for the Boltzmann-Enskog process}
We have now established the necessary estimates on the Fourier transform of the approximating process to discuss the beneficial estimate stated in Lemma \ref{approx-est-lem}. The benefit of the Fourier transform estimates will be seen in calculations related to the approximating process $(R^\varepsilon_t,V^\varepsilon_t)$. We now prove Lemma \ref{approx-est-lem} before finally proving the main theorem.\\
\textbf{Proof of Lemma \ref{approx-est-lem} }\\
The estimates established in earlier sections can now be utilized to give the result of Lemma \ref{approx-est-lem}.
\begin{proof} 
Take $\varepsilon \in (0,1)$. Take $0 < t_0 \leq t - \varepsilon \leq t \leq t_1$, $\phi \in L^\infty(\mathbb{R}^3)$ and $h \in \mathbb{R}^3$. Recall our definitions of the processes $U_t^\varepsilon$ and $W_t^\varepsilon$ earlier. Then
\begin{equation*}
\begin{split}
        & |E[\phi(V_t^\varepsilon + h) - \phi(V_t^\varepsilon)]|\\
        & = |E[\phi(U_t^\varepsilon + W_t^\varepsilon + h) - \phi(U_t^\varepsilon + W_t^\varepsilon)]|\\
        & = |E[E[\phi(U_t^\varepsilon + W_t^\varepsilon + h) - \phi(U_t^\varepsilon + W_t^\varepsilon)|\mathcal{F}_{t-\varepsilon}]]|\\
        & |E\left[\int_{\mathbb{R}^3}\phi(x + W_t^\varepsilon + h) - \phi(x + W_t^\varepsilon)g_{\varepsilon,t,V_{t-\varepsilon},R_{t-\varepsilon}}(x)dx\right]|\\
        & = |E\left[\int_{\mathbb{R}^3} \phi(x + W_t^\varepsilon)[g_{\varepsilon,t,V_{t-\varepsilon},R_{t-\varepsilon}}(x - h) - g_{\varepsilon,t,V_{t-\varepsilon},R_{t-\varepsilon}}(x)]dx\right]|\\
        & \leq  \norm{\phi}_{L^\infty(\mathbb{R}^3)}|h|E[\norm{\nabla g_{\varepsilon,t,V_{t-\varepsilon},R_{t-\varepsilon}}}_{L^1(\mathbb{R}^3)}].
\end{split}
\end{equation*}
where we use the following calculation which is immediate using the multivariate mean value theorem.
\begin{equation*}
    \begin{split}
    &\int_{\mathbb{R}^3}|g(x - h) - g(x)|\,dx\\
    & \leq \int_{\mathbb{R}^3}\int_0^1|h||\nabla g(x - uh)|\,du\,dx\\
    & \leq |h|\int_0^1\norm{\nabla g}_{L^1(\mathbb{R}^3)}\,du\\
    & = |h|\norm{\nabla g}_{L^1(\mathbb{R}^3)}
    \end{split}    
\end{equation*}
Thus, for $\gamma \in (0,2)$, by the above work and Lemma \ref{Upper-BD},
\begin{equation*}
    |E[\phi(V_t^\varepsilon + h) - \phi(V_t^\varepsilon)]| \leq C_{t_0,t_1}\norm{\phi}_{L^\infty(\mathbb{R}^3)}\varepsilon^{-1/\nu}E[(1 + |V_{t - \varepsilon}|)^{\gamma + 4}].
\end{equation*}
The right side of this inequality is finite as $E|V_{t - \varepsilon}|^{\gamma + 4} < \infty$. Thus,
\begin{equation*}
    |E[\phi(V_t^\varepsilon + h) - \phi(V_t^\varepsilon)]| \leq C_{t_0,t_1}\norm{\phi}_{L^\infty(\mathbb{R}^3)}|h|\varepsilon^{-1/\nu}
\end{equation*}
\end{proof}
\noindent \textbf{Existence of Density Theorem}\\
Finally, the proof for the existence of a density function for the distribution of $V$ at each finite time $t$ can be stated and proven. The estimate of Lemma \ref{approx-est-lem} has given a needed window for the required power in equation (\ref{deb-rom-eqn}).
\begin{thm}
    Let $(R,V)$ be an Boltzmann-Enskog process as defined in Theorem \ref{sund-estimate-thm} with $\gamma \in (0,2)$. Assume the assumptions on $\alpha$,$\beta$ and $\sigma$ in section \ref{assume} are satisfied. Let $f_0 \in \mathcal{P}_2(\mathbb{R}^6)$ not be a Dirac mass. Then $V_t$ has a density which resides in a Besov space. 
\end{thm}
\begin{proof}
    Fix $t > 0$. We want to apply Theorem \ref{deb-rom-exist}. Let $h \in \mathbb{R}^3$ with $|h| \leq 1$ and $\phi \in C^\alpha_b(\mathbb{R}^3)$ for $\alpha \in (0,1)$. Define
    \begin{equation*}
        I^\phi_{t,h} = |E[(\phi(V_t + h) - \phi(V_t)]|
    \end{equation*}
    Recall the approximated process $V^\varepsilon_t$, then we have 
    \begin{equation*}
        \begin{split}
            I^\phi_{t,h} & \leq |E[(\phi(V_t + h) - \phi(V_t^\varepsilon + h)]| + |E[(\phi(V_t^\varepsilon) - \phi(V_t)]|\\
            & + |E[(\phi(V_t^\varepsilon + h) - \phi(V_t^\varepsilon)]|\\
            & \leq 2\norm{\phi}_{C^\alpha_b(\mathbb{R}^3)}E|V_t - V_t^\varepsilon|^\alpha + C_{t}\norm{\phi}_{L^\infty(\mathbb{R}^3)}\varepsilon^{-1/\nu}|h|\\
            & \leq C_t\norm{\phi}_{C^\alpha_b(\mathbb{R}^3)}[E|V_t - V_t^\varepsilon|^\alpha + \varepsilon^{-1/\nu}|h|]
        \end{split}
    \end{equation*}
    where we applied the upper estimate from Lemma \ref{approx-est-lem} and that $\norm{\phi}_{L^\infty(\mathbb{R}^3)} \leq \norm{\phi}_{C^\alpha_b(\mathbb{R}^3)}$. Now, assuming $\gamma \in (0,1)$, we use Jensen's inequality and the estimate from Theorem \ref{approx-def}, to give
    \begin{equation*}
    I^\phi_{t,h} \leq C_t\norm{\phi}_{C^\alpha_b(\mathbb{R}^3)}[\varepsilon^{\frac{2\gamma + 1}{\gamma + 1}\alpha} + \varepsilon^{-1/\nu}|h|]
    \end{equation*}
    Provided we take $c > \frac{\gamma + 1}{2\gamma + 1}$ and $c < \nu$, then we can let $\varepsilon = |h|^{c}$. This gives the upper bound of 
    \begin{equation*}
        I^\phi_{t,h} \leq C_t\norm{\phi}_{C^\alpha_b(\mathbb{R}^3)}[|h|^{\frac{2\gamma + 1}{\gamma + 1}c\alpha} + |h|^{1 -c/\nu}].
    \end{equation*}
    Thus, we can conclude that there is some $c$ we can raise so that 
    \begin{equation*}
        I^\phi_{t,h} \leq \kappa\norm{\phi}_{C^\alpha_b(\mathbb{R}^3)}|h|^a
    \end{equation*}
    for some positive power $a$ where $\alpha < a < 1$. If $\gamma \in [1,2)$, we can instead bound above by 
    \begin{equation*}
        I^\phi_{t,h} \leq C_t\norm{\phi}_{C^\alpha_b(\mathbb{R}^3)}[|h|^{\frac{\gamma + 2}{\gamma + 1}c\alpha} + |h|^{1 -c/\nu}]
    \end{equation*}
    using the other upper bound noted in Proposition \ref{approx-def}. In either case, we can conclude there exists a density using Theorem \ref{deb-rom-exist}. 
\end{proof}

\begin{appendix}\noindent {\Large \bf Appendix}
	
\begin{thm}
Let $\{f_t\}_{t \geq 0}$ be a solution of the Boltzmann-Enskog equation \eqref{EQ:03}. Let the following assumptions hold 
\begin{itemize}
	\item \begin{equation} \label{C1}
		\sup_{t}  \int_{\R^{6}} |v|^3\, | v \cdot \nabla_r f_t(r,v)| drdv < \infty \end{equation}
	\item \begin{equation} \label{C2} \lim_{|r|\to \infty} f(t,r, v) =0 \end{equation}
\end{itemize}
then \eqref{conv-laws}
%\eqref{consmass}, \eqref{consmomentum}, \eqref{consenergy} 
hold.
\end{thm}	
\begin{proof}
	Let $\{f_t\}_{t \geq 0}$ be a solution of the Boltzmann -Enskog equation \eqref{EQ:03}. For any  $A,C\in \R$ and $B\in \R^3$ consider the function
	\begin{equation*}
		\Phi(v)= A+B \cdot v +C |v|^2, \quad v \in \R^3
	\end{equation*}
Then 
\begin{equation}\label{Qinv}
\int_{\R^6}\Phi(v) \mathcal{Q}(f_t,f_t)(r,v) dr dv=0 
\end{equation}
The proof of \eqref{Qinv} is identical to the proof proposed by Boltzmann  to prove  that\\ $\int_{\R^6}\Phi(v) {Q}(f,f)(t, r,v) dr dv=0$ for the non -linear term  appearing in the Boltzmann equation, presented in Sect. 3.1 of the monograph by Cercignani, Illner and Pulvirenti \cite{CIPbook}.\\
The weak formulation of the Boltzmann -Enskog equation provides in particular that
\begin{align*}
&\hspace{5cm}\int_{\R^6}\Phi(v) f_t(r,v) drdv =\\
&\int_{\R^6}\Phi(v) f_0(r,v) drdv - \int_{\R^6}\Phi(v) \, v \cdot \nabla_r f_t(r,v) drdv +\int_{\R^6}\Phi(v) \mathcal{Q}(f_t,f_t)(r,v) dr dv
\end{align*}
By assumption \eqref{C1} the second term in the r.h.s. can be integrated by parts and is zero due to assumption \eqref{C2}. The statement is then a consequence of \eqref{Qinv}.
\end{proof}
In \cite{sund} the following Corollary is proven.
\begin{corollary} [Corollary 2.6 \cite{sund}] 
%In fact, in Corollary 2.6 of \cite{FRS1} it is proven that any solution $\{f_t\}_{t\geq 0 }$ to (\ref{EQ:03}) satisfies \eqref{consmass} and that
Let $\{f_t\}_{t \geq 0}$ be a solution of the Boltzmann -Enskog equation \eqref{EQ:03}.
Under the sufficient condition  
	\begin{equation}\label{int third moment}
		\int_0^T  \int_{\R^{6}}|v|^3f_t(r,v)drdvdt < \infty \quad \forall \,T>0
	\end{equation}
    the  solution $\{f_t\}_{t\geq 0 }$ is conservative, i.e satisfies \eqref{conv-laws}. %\eqref{consmomentum} and \eqref{consenergy},  as well. \\
\end{corollary}

	\end{appendix}

\end{document}